\documentclass[10pt]{article}
\usepackage{amssymb}
\usepackage{url}
\usepackage[dvips]{graphicx}
\usepackage{graphics}
\usepackage[dvips]{color}
\usepackage{xcolor}
\usepackage[figuresright]{rotating}
\usepackage{multirow}
\usepackage[top=2cm, bottom=2cm, left=2cm, right=2cm]{geometry}
\usepackage{algorithm}
\usepackage{algorithmicx}
\usepackage{algpseudocode}
\usepackage{algpseudocode}
\usepackage{amsmath}
\usepackage{amsthm}
\usepackage{amsmath}
\usepackage{amsfonts}
\usepackage{epic,eepic}
\usepackage{epsfig}
\usepackage{booktabs} 
\usepackage[toc,page,title,titletoc,header]{appendix}
\usepackage[font={small}]{caption}
\newtheorem{thm}{Theorem}[section]
\newtheorem{lem}[thm]{Lemma}
\newtheorem{cor}[thm]{Corollary}
\newtheorem{prop}[thm]{Proposition}
\newtheorem{rem}[thm]{Remark}

\newtheorem{con}[thm]{Condition}

\newtheorem{asmp}[thm]{Assumption}

\newtheorem{dfn}[thm]{Definition}

\def\dg{{\rm diag}}

\def\nl{{\rm Null}}

\def\ra{{\rm Range}}

\newcommand\norm[1]{\left\lVert#1\right\rVert}

\usepackage{amsfonts}

\date{\today}
\author{Rujun Jiang\thanks{School of Data Science, Fudan University, Shanghai, China, rjjiang@fudan.edu.cn}
\and Duan Li\thanks{Department of Management Sciences, City University of Hong Kong, Hong Kong,  dli226@cityu.edu.hk}}

\title{Novel Reformulations and Efficient Algorithms for the Generalized Trust Region Subproblem}
\usepackage{varioref}

\begin{document}

\maketitle
\begin{abstract}
We present a new solution framework to solve the generalized trust region subproblem (GTRS)  of minimizing a quadratic objective  over a quadratic constraint. More specifically, we derive a convex quadratic reformulation (CQR)  via minimizing a linear objective over two convex quadratic constraints for the GTRS.
We show that an optimal solution of the GTRS can be recovered from an optimal solution of the CQR. We further prove that this CQR is equivalent to minimizing the maximum of the two convex quadratic functions derived from the CQR for the case under our investigation. Although the latter minimax problem is nonsmooth, it is well-structured and convex. We thus develop two  steepest descent algorithms corresponding to two different line search rules. We prove for both algorithms their global sublinear convergence rates. We also obtain a local linear convergence rate  of the first algorithm by estimating the  Kurdyka-{\L}ojasiewicz exponent at any optimal solution under mild conditions.  We finally demonstrate the efficiency of our algorithms in our numerical experiments.
\end{abstract}

\section{Introduction}
We consider the following generalized trust region subproblem (GTRS),
\begin{eqnarray}
\rm(P) &\min&f_1(x):=\frac{1}{2}x^\top Q_1x+b_{1}^\top x\notag\\
&\rm s.t.&f_2(x):=\frac{1}{2}x^\top Q_2x+b_{2}^\top x+c\leq 0,\notag
\end{eqnarray}
where $Q_1$ and $Q_2$ are $n\times n$ symmetric matrices (not necessary to be positive semidefinite),
$b_{1},b_2\in\mathbb{R}^n$ and $c\in\mathbb{R}$.

Problem (P) is known as the generalized trust region subproblem (GTRS)  \cite{stern1995indefinite, pong2014generalized}. When $Q_2$ is an identity matrix $I$ and $b_{2}=0$, $c=-1/2$, problem (P) reduces to the classical trust region subproblem (TRS). The TRS first arose in the trust region method for nonlinear optimization \cite{conn2000trust,yuan2015recent}, and  has found many  applications including robust optimization \cite{ben2009robust} and the least square problems \cite{zhang2010derivativefree}.
As a generalization, the GTRS also admits its own applications such as time of arrival problems \cite{hmam2010quadratic} and subproblems of consensus ADMM in signal processing \cite{huang2016consensus}.  Over the past two decades,
numerous solution methods have been developed for TRS (see
\cite{More1983Computing,martinez1994local,ye1992new,rendl1997semidefinite,hazan2016a,gould2010solving,adachi2017solving} and references therein).

Various methods have been developed for  solving the GTRS under various assumptions
(see \cite{more1993generalizations,stern1995indefinite,ben1996hidden,sturm2003cones,feng2012duality,pong2014generalized,adachi2016eigenvaluebased} and references therein).
Although it appears being nonconvex, the GTRS essentially enjoys its hidden convexity. The GTRS
can be solved via a semidefinite programming (SDP) reformulation, due to the celebrated S-lemma  \cite{polik2007survey}, which was first established in
 \cite{yakubovich1971s}. However, suffering from relatively large computational complexity, the SDP algorithm is not practical for large-scale  applications.
To overcome this difficulty, several recent papers \cite{jeyakumar2014trust,burer2016how,ho2017second} demonstrated that the TRS admits a second order cone programming (SOCP) reformulation.  Ben-Tal and den Hertog  \cite{ben2014hidden} further showed an SOCP reformulation for the GTRS under a  simultaneously diagonalizing (SD) procedure of the quadratic forms.   Jiang et al. \cite{jiang2017socp} derived an SOCP reformulation for the GTRS when the problem has a finite optimal value and further derived a closed form solution when the SD condition fails.
On the other hand, there is rich literature on iterative algorithms to solve the GTRS directly under mild conditions, for example, \cite{more1993generalizations, stern1995indefinite,pong2014generalized,salahi2016efficient}.  Pong and Wolkowicz proposed an efficient algorithm based on minimum generalized eigenvalue of a parameterized
matrix pencil for the GTRS, which extended  the results in \cite{fortin2004trust} and \cite{rendl1997semidefinite} for the TRS.
 Salahi and Taati \cite{salahi2016efficient}  also derived  a diagonalization-based algorithm under the SD condition of the quadratic forms. Recently, Adachi and Nakatsukasa \cite{adachi2016eigenvaluebased} also developed a novel eigenvalue-based algorithm to solve the GTRS.

Our main contribution in this paper is to propose a novel convex quadratic reformulation  (CQR) for the GTRS that is simpler than \cite{ben2014hidden,jiang2017socp}  and further a minimax problem reformulation and develop an efficient algorithm  to solve the minimax problem reformulation. Numerical results demonstrate that our method outperforms all the existing methods in the literature for sparse problem instances. We acknowledge that our CQR was inspired by  the following CQR in Flippo and Janson \cite{flippo1996duality} for the TRS,
\begin{equation}\label{trscqr}
\min_{x}\{\frac{1}{2}x^\top(Q_1-\lambda_{\min}(Q_1)I)x+b_{1}^\top x+\frac{1}{2}\lambda_{\min}(Q_1): ~x^\top x\leq 1\},
\end{equation}
where $\lambda_{\min}(Q_1)$ is the smallest eigenvalue of matrix $Q_1$. Unfortunately, this CQR was underappreciated in that time. Recently, people rediscovered this result; Wang and Xia \cite{wang2016linear}
 and Ho-Nguyen and Kilinc-Karzan \cite{ho2017second} presented a linear time algorithm to solve the TRS by applying Nesterov's accelerated gradient descent algorithm to (\ref{trscqr}).  We, instead, rewrite  the epigraph reformulation for (\ref{trscqr}) as follows,
$$\min_{x,t}\{t:~\frac{1}{2}x^\top(Q_1-\lambda_{\min}(Q_1)I)x+b_{1}^\top x+\frac{1}{2}\lambda_{\min}(Q_1)\leq t, ~x^\top x\leq 1\}.$$
Motivated by the above reformulation, we demonstrate that the GTRS is equivalent to exact one of the following two convex quadratic  reformulations under two different conditions,
\begin{eqnarray*}
{\rm(P_1)}&\min_{x,t}&\{t:~h_1(x)\leq t,~ h_2(x)\leq t\},\\
{\rm(P_2)}&\min_{x,t}&\{t:~h_3(x)\leq t, f_2(x)\leq 0\},
\end{eqnarray*}
where $h_1(x)$, $h_2(x)$ and $h_3(x)$, which will be defined later in Theorem \ref{mainthm},   and $f_2(x)$ defined in problem (P),  are  convex but possibly not strongly convex, quadratic functions. To our best knowledge, our proposed CQRs are derived the first time for the GTRS.
The  reformulation $\rm(P_2)$ only occurs when the quadratic constraint is convex and thus can be solved by a slight modification of \cite{wang2016linear,ho2017second} in the accelerated gradient projection method by projecting, in each iteration, the current solution to the ellipsoid instead of the unit ball in the TRS case.

In this paper we focus on the problem reformulation $\rm(P_1)$. Although our CQR can be solved  as an SOCP problem \cite{ben2001lectures}, it is not efficient when the problem size is large. Our main contribution is based on a recognition that problem $\rm(P_1)$ is equivalent to minimizing the maximum of the two convex quadratic functions in  $\rm(P_1)$,
$${\rm(M)}~\min H(x):=\max \{h_1(x),h_2(x)\}.$$
 We further derive efficient algorithms to solve the above minimax problem. To the best of our knowledge, the current literature lacks studies on such a problem formulation for a large scale setting except using a black box subgradient method with an $O(1/\epsilon^2)$ convergence rate \cite{boyd2006subgradient}, which is really slow. Note that  Section 2.3 in Nesterov's book \cite{nesterov2003introductory} presents a gradient based method with linear convergence rate for solving the minimization  problem $(M)$ under the condition that both $h_1(x)$ and $h_2(x)$ are strongly convex. However, Nesterov's algorithms cannot be applied to solve our problem since in our problem setting at least one function of $h_1(x)$ and $h_2(x)$ is not strongly convex.
By using the special structure of problem (M),  we derive  a steepest descent method in Section 3. More specifically, we choose either the negative gradient when the current point is smooth,  or  a  vector in the subgradient set with the smallest norm (the steepest descent direction) when the current point is nonsmooth  as the descent direction, and derive two steepest descent algorithms with two different line search rules accordingly. In the first algorithm we choose a special step size, and in the second algorithm we propose
a modified Armijo line search  rule.
We also prove the global sublinear convergence rate for both algorithms.
 The
first algorithm even admits a global convergence rate of $O(1/\epsilon)$, in the same order as the gradient
descent algorithm, which is faster than the subgradient method. In addition, we demonstrate that the first algorithm also admits  a local linear convergence rate,  by a delicate analysis on the Kurdyka-{\L}ojasiewicz (KL) \cite{attouch2009convergence,bolte2015error,liu2016quadratic,gao2016ojasiewicz} property for problem (M).
We illustrate in our numerical experiments the efficiency of the proposed algorithms  when compared with the state-of-the-art methods for GTRS in the literature.



The rest of this paper is organized as follows. In Section 2, we derive an explicit CQR for problem (P) under different conditions and show how to recover an optimal solution of problem (P) from that of the CQR.
In Section 3, we reformulate the CQR to a convex nonsmooth unconstrained minimax problem and derive two efficient solution algorithms. We provide convergence analysis for both algorithms. In Section 4, we demonstrate the efficiency of our algorithms from our numerical experiments. We conclude our paper in Section 5.

\textbf{Notations}  We use $v(\cdot)$ to denote the optimal value of problem $(\cdot)$.  The matrix transpose of matrix $A$ is denoted by $A^\top$ and  inverse of matrix $A$ by $A^{-1}$, respectively.
\section{Convex quadratic reformulation}
In this section, we derive a novel convex quadratic reformulation for problem (P). To avoid some trivial cases, we assume, w.o.l.g., the Slater condition holds for problem (P), i.e., there exists at least one interior feasible point. When both $f_1(x)$ and $f_2(x)$ are convex,  problem (P) is already a convex quadratic problem.
Hence, w.l.o.g., let us assume that not both $f_1(x)$ and $f_2(x)$ are convex.
We need to introduce the following conditions to exclude some unbounded cases.
\begin{asmp}\label{poreg}
The set $I_{PSD}:=\{\lambda:Q_1+\lambda Q_2\succeq0\}\cap \mathbb{R}_+$ is not empty, where $\mathbb{R}_+$ is the nonnegative orthant.\end{asmp}
\begin{asmp}\label{comnull}
The common null space of $Q_1$ and $Q_2$ is  trivial, i.e., $\nl(Q_1)\cap\nl(Q_2)=\{0\}$.
\end{asmp}

Before introducing our CQR, let us first recall the celebrated S-lemma by defining $\tilde f_1(x)=f_1(x)+\gamma$ with an arbitrary constant $\gamma\in \mathbb{R}$.
\begin{lem}[S-lemma \cite{yakubovich1971s,polik2007survey}]\label{slemma}
The following two statements are equivalent:\\
1. The system of $\tilde f_1(x)<0$ and $f_2(x)\leq0$ is not solvable;\\
2. There exists $\mu\geq0$ such that $\tilde f_1(x)+\mu f_2(x)\geq0$ for all $x\in\mathbb{R}^n$.
\end{lem}

Using the S-lemma, the following lemma shows a necessary and sufficient condition under which problem (P) is bounded from below.\begin{lem}[\cite{Hsia2014revisit}]\label{nscon}
Problem (P) is bounded from below if and only if the following system has a solution for $\lambda$:
\begin{equation*}\label{condxia}Q_1+\lambda Q_2\succeq0,~\lambda\geq0,~b_1+\lambda b_2\in\ra(Q_1+\lambda Q_2).\end{equation*}
\end{lem}


We make Assumption \ref{comnull} without loss of generality, because otherwise we can prove an unboundedness from below of the problem  (see, e.g., \cite{jiang2017socp} and \cite{adachi2016eigenvaluebased}). Under Assumption \ref{comnull}, if Assumption \ref{poreg} fails, there exists no nonnegative $\lambda$ such that $Q_1+\lambda Q_2\succeq0$ and problem (P) is unbounded from below due to Lemma \ref{nscon}.  So both Assumptions \ref{poreg} and  \ref{comnull} are made without loss of generality.

It has been shown in \cite{more1993generalizations} that $\{\lambda:Q_1+\lambda Q_2\succeq0\}$ is an interval and thus $\{\lambda:Q_1+\lambda Q_2\succeq0\}\cap \mathbb{R}_+$ is also an interval (if not empty). Under Assumptions \ref{poreg} and \ref{comnull},   we have the following three cases for $I_{PSD}$.
\begin{con}\label{con1}
The set $I_{PSD}=[\lambda_1,\lambda_2]$ with $\lambda_1<\lambda_2$.
\end{con}
\begin{con}\label{con2}
The set  $I_{PSD}=[\lambda_3,\infty)$.
\end{con}
\begin{con}\label{con3}
The set $I_{PSD}=\{\lambda_4\}$ is a singleton.
\end{con}

Note that Condition \ref{con2} occurs only when $Q_{2}$ is positive semidefinite. Under Condition \ref{con3}, $Q_1$ and  $Q_2$ may not be SD and may have $2\times2$ block pairs in a canonical form under congruence \cite{jiang2017socp}. In this case, when $\lambda$ is given,  the authors in \cite{jiang2017socp} showed how to recover an optimal solution if the optimal solution is attainable, and how to obtain an $\epsilon$ optimal solution if the optimal solution is unattainable.
So in the following, we mainly focus on the cases where either Condition \ref{con1} or \ref{con2} is satisfied.\begin{lem}\label{boundedness}Under Condition \ref{con1} or \ref{con2}, problem (P) is bounded from below.
\end{lem}
\begin{proof}
Under Condition \ref{con1} or \ref{con2}, there exists $\lambda_0$ such that $Q_1+\lambda_{0} Q_2\succ0$ and $\lambda_0\geq0$, which further implies $b_1+\lambda_{0} b_2\in\ra(Q_1+\lambda_{0} Q_2)$ as $Q_1+\lambda Q_2$ is nonsingular. With Lemma \ref{nscon}, we complete the proof.
\end{proof}
\subsection{Convex quadratic reformulation for GTRS}
It is obvious that problem (P) is equivalent to its epigraph reformulation as follows,
\begin{eqnarray*}
{\rm( P_{0})} ~\min\{t:~f_1(x)\leq t,~f_2(x)\leq 0\}.
\end{eqnarray*}
To this end, we are ready to present the main result of this section.
\begin{thm}\label{mainthm}Under Assumption \ref{poreg}, by defining $h_i(x)=f_1(x)+\lambda_if_2(x),~i=1,2,3$, we can reformulate problem $\rm(P)$  to a convex quadratic problem under Conditions \ref{con1} and \ref{con2}, respectively:
\begin{enumerate}
\item  Under Condition \ref{con1}, problem $\rm(P)$ is equivalent to the following convex quadratic problem, $${\rm(P_1)}~\min_{x,t}\{t:~h_1(x)\leq t,~ h_{2}(x)\leq t\};$$
\item Under Condition \ref{con2},  problem $\rm(P)$ is equivalent to the following convex quadratic problem,
$${\rm(P_2)}~\min_{x,t}\{t:~h_{3}(x)\leq t, ~f_2(x)\leq 0\}=\min_{x}\{h_{3}(x):~ f_2(x)\leq 0\}.$$
\end{enumerate}
\end{thm}
\begin{proof}
Let us first consider the case where Condition \ref{con1} holds. Due to Lemma \ref{boundedness}, $\rm(P_1)$ is bounded from below. Together with the assumed Slater conditions,  problem $\rm(P_1)$  admits the same optimal value as its  Lagrangian dual \cite{ben2001lectures}. Due to the S-lemma, problem (P) also has the same optimal value as its Lagrangian dual \cite{sturm2003cones},
$${\rm(D)}~\max_{\mu\geq0}\min_{x} f_1(x)+\mu f_2(x).$$
Under Condition \ref{con1}, i.e., $I_{PSD}=[\lambda_1,\lambda_2]$ with $\lambda_1<\lambda_2$, it is easy to show that $\rm(P_1)$ is a relaxation of $\rm(P_0)$ since they have the same objective function and the feasible region of $\rm(P_1)$ contains that of $\rm(P_0)$ (note that $f_1\leq t$ and $f_2\leq0$ imply that $f_1(x)-t+uf_2(x)\leq0$ for all $u\geq0$). Thus, \begin{equation}\label{ineq1}
v({\rm P_1})\leq v({\rm P_0})=v(\rm P).
\end{equation}
The Lagrangian dual problem of $\rm(P_1)$ is
$${\rm(D_1)}~\max_{s_1,s_2\geq0}\min_{x,t} t+(s_1+s_2)(f_1(x)-t)+(\lambda_1s_1+\lambda_2s_2) f_2(x).$$
For any primal and dual optimal solution pair $(x^*,u^*)$ of (P) and (D), due to  $u^*\in[\lambda_1,\lambda_2]$ as $Q_1+\mu^{*} Q_2\succeq0$ from Lemma \ref{nscon}, we can always find a convex combination $\lambda_1\bar s_1+\lambda_2\bar s_2=\mu^*$ with $\bar s_1+\bar s_2=1$. Hence $(x^*,\bar s,t)$, with an arbitrary $t$, is a feasible solution  to  $\rm(D_1)$ and  the objective value of problem $(D_1)$ at  $(x^*,\bar s,t)$ is the same with the optimal value of $\rm(D)$. This in turn implies
\begin{equation}\label{ineq2}
v{\rm(D_1)}\geq v{\rm(D)}.
\end{equation}
Since $\rm(P_1)$ is convex and Slater condition is satisfied (because $\rm(P_1)$ is a relaxation of (P) and Slater condition is assumed for (P)), $v{\rm(P_1)}=v{\rm(D_1)}$. Finally, by combining (\ref{ineq1}) and (\ref{ineq2}), we have $v{\rm(P_1)}=v{\rm(D_1)}\geq v{\rm(D)}=v{\rm(P)}=v{\rm(P_0)}\geq v{\rm(P_1)}$. So all inequalities above become equalities and thus $\rm(P_1)$ is equivalent to (P).

Statement 2 can be proved in a similar way and is thus omitted.
\end{proof}
\begin{rem}Reformulation $\rm(P_2)$  generalizes the approaches in \cite{flippo1996duality,wang2016linear,ho2017second} for the classical TRS with the unit ball constraint to the GTRS with a general convex quadratic constraint.
\end{rem}

 To our best knowledge, there is no method in the literature to compute $\lambda_1$ and $\lambda_2$  in Condition \ref{con1} for general $Q_1$ and $Q_2$. However, there exist  efficient methods in the literature to compute   $\lambda_1$ and $\lambda_2$  when a $\lambda_0$ is given such that $Q_1+\lambda_0Q_2\succ0$ is satisfied. More specifically, the method mentioned in Section 2.4.1 in \cite{adachi2016eigenvaluebased} gives a way to compute $\lambda_1$ and $\lambda_2$: first detect a $\lambda_0$ such that $Q_{0} 瞿繙:=Q_1+\lambda_0Q_2\succ0$, and then compute $\lambda_1$ and $\lambda_2$ by some generalized eigenvalues  for a definite matrix pencil that are nearest to 0. Please refer to \cite{guo2009improved} for one of the state-of-the-art methods for detecting $\lambda_0$. We can also find another iterative method in Section 5 \cite{more1993generalizations} to compute $\lambda_0\in {\rm int}(I_{PSD})$ by reducing the length of an interval $[\bar \lambda_1,\bar \lambda_2]\supset I_{PSD}$.
We next report our new method to compute $\lambda_1$ and $\lambda_2$, which is motivated by \cite{pong2014generalized}.
 Our first step is also to find a $\lambda_0$ such that  $Q_{0}:=Q_1+\lambda_0Q_2\succ0$. Then we compute the maximum generalized eigenvalues for $Q_2+\mu Q_{0}$ and $-Q_2+\mu Q_{0}$, denoted by $u_1$ and $u_2$, respectively. Note that both  $u_1>0$ and $u_2>0$   due to  $Q_0\succ0$ and $Q_2$ has a negative eigenvalue.  So we have
$$Q_1+(\frac{1}{u_1}+\lambda_0) Q_2\succeq0~ \text{ and }~ Q_1+(-\frac{1}{u_2}+\lambda_0) Q_2\succeq0.$$
Thus $Q_1+\eta Q_2\succeq0$ for all $\eta\in[\lambda_0-\frac{1}{u_2},\lambda_0+\frac{1}{u_1}]$, which implies $\lambda_1=\lambda_0-\frac{1}{u_2} $ and $\lambda_2=\lambda_0+\frac{1}{u_1}$.
In particular, when one of $Q_1$ and $Q_2$ is positive definite, we can skip the step of detecting the definiteness, which would save significant time in implementation.

In fact, when $\lambda_0$ is given, we only need to compute one extreme eigenvalues, either $\lambda_1$ or $\lambda_2$, to obtain our convex quadratic reformulation.  Define $x(\lambda)=-(Q_1+\lambda Q_2)^{-1}(b_1+\lambda b_2)$ for all $\lambda\in{\rm int}(I_{PSD})$ and define $\gamma(\lambda)=f_{2}(x(\lambda))$.
 After we have computed $\lambda_0$ such that $\lambda_0\in{\rm int}(I_{PSD})$, under Assumption \ref{comnull},  we further have $Q_1+\lambda_0Q_2\succ0$, which makes $(Q_1+\lambda Q_2)^{-1}$ well defined.
In fact, there are Newton type methods in the literature (e.g., \cite{more1993generalizations}) for solving the GTRS by finding the optimal $\lambda$ through  $\gamma(\lambda)=0$.  However, each step in \cite{more1993generalizations} involves solving a linear system $-(Q_1+\lambda Q_2)^{-1}(b_1+b_2)$, which is time consuming for high-dimension settings. Moreover, the Newton's method does not converge in the so called  hard case\footnote{The definition here follows \cite{more1993generalizations}. In fact, the definitions of hard case and easy case of the GTRS are similar to those of the TRS. More specifically, if the null space of the Hessian matrix, $Q_{1}+\lambda^* Q_{2}$, with  $\lambda^{*} $ being the optimal Lagrangian multiplier of problem (P), is orthogonal to $b_1+\lambda^{*}  b_2$, we are in the hard case; otherwise we are in the easy case.}. On the other hand, for easy case, an initial $\lambda$ in  $I_{PSD}$ is also needed as a safeguard to guarantee the positive definiteness of $Q_1+\lambda Q_2$ \cite{more1993generalizations}.
It is shown in \cite{more1993generalizations} that $\gamma(\lambda)$ is either a strictly decreasing function or a constant in ${\rm int(}I_{PSD})$. Following  \cite{adachi2016eigenvaluebased}, we have the following three cases: if $\gamma(\lambda_0)>0$, the optimal $\lambda^*$ locates in $[\lambda_0,\lambda_2]$; if $\gamma(\lambda_0)=0$, $x(\lambda_0)$ is an optimal solution; and if $\gamma(\lambda_0)<0$, the optimal $\lambda^*$ locates in $[\lambda_1,\lambda_0]$.
Hence we have the following corollary, whose proof is similar to that in Theorem \ref{mainthm} and thus omitted.\begin{cor}\label{cqr2}
Assume that Assumption \ref{poreg} holds and define $h_i(x)=f_1(x)+\lambda_if_2(x), i = 0, 1, 2$.
 Under Condition \ref{con1}, the following results hold true.
\begin{enumerate}
\item If $\gamma(\lambda_0)>0$, problem $\rm(P)$ is equivalent to the following convex quadratic problem,
$${\rm(\overline{P}_1)}~\min_{x,t}\{t:~h_0(x)\leq t,~ h_{2}(x)\leq t\}.$$
\item  If $\gamma(\lambda_0)=0$, $x(\lambda_0)=-(Q_1+\lambda_{0} Q_2)^{-1}(b_1+\lambda_{0} b_2)$  is the optimal solution.
\item If $\gamma(\lambda_0)<0$, problem $\rm(P)$ is equivalent to the following convex quadratic problem,
$${\rm(\underline{P}_1)}~\min_{x,t}\{t:~h_1(x)\leq t,~ h_{0}(x)\leq t\}.$$\end{enumerate}
\end{cor}
Since both ${\rm(\overline{P}_1)}$  and $\rm(\underline{P}_1)$ have a similar form to $\rm (P_1)$ and can be solved in a way similar to the solution approach for $\rm(P_1)$, we only discuss how to solve $\rm (P_1)$ in the following.
\subsection{Recovery of optimal solutions}
In this subsection, we will discuss the recovery of an optimal solution to problem (P)\ from an optimal solution to reformulation $\rm(P_1)$.
Before
that, we first introduce the following lemma. Let us assume from now on $h_i(x)=\frac{1}{2}x^\top A_ix+a_i^\top x+r_i$, $i=1,2$. \begin{lem}\label{SDinterval}
If Condition \ref{con1} holds, $A_1$ and $A_2$ are simultaneously diagonalizable. Moreover, we have   $d^\top A_1d>0$  for all nonzero vector
$d\in\nl(A_{2})$.\end{lem}
\begin{proof}
Note that Condition \ref{con1} and Assumption \ref{comnull} imply that $Q_1+\frac{\lambda_1+\lambda_2}{2}Q_2\succ0 $, i.e.,  $\frac{A_1+A_2}{2}\succ0$. Let $A_0=\frac{A_1+A_2}{2}$ and $A_{0}=L^\top L$ be its Cholesky decomposition, where $L$ is a nonsingular symmetric matrix. Also let $(L^{-1})^\top A_1L^{-1}=P^\top DP$ be the spectral decomposition, where $P$ is an orthogonal matrix and $ D$ is a diagonal matrix. Then we have  $(L^{-1}P^{-1})^\top A_1L^{-1}P^{-1}=D$ and $$(L^{-1}P^{-1})^\top A_2L^{-1}P^{-1}=(L^{-1}P^{-1})^\top A_0L^{-1}P^{-1}-(L^{-1}P^{-1})^\top A_1L^{-1}P^{-1}=I-D.$$Hence $A_1$ and $A_2$ are simultaneously diagonalizable by the congruent matrix $L^{-1}P^{-1}$.

Now let us assume $S=L^{-1}P^{-1}$ and thus $S^\top A_1S=\dg(p_1,\ldots,p_n) $ and $S^\top A_2S=\dg(q_1,\ldots,q_n)$ are both diagonal matrices. Define  $K=\{i:q_i=0,i=1,\ldots,n\}$.  Since $A_1+A_2\succ0$, $p_i>0$ for all $i\in K$. Let $e_i$ be  the $n$-dimensional
vector with $i$th entry being 1 and all others being 0s. We have $(Se_i)^\top A_{1}Se_i=p_i>0$ for all $i\in K$. On the other hand, $A_2Se_i=0$ for all $i\in K$. Hence    $d^\top A_1d>0$ holds for all nonzero vector
$d\in\nl(A_{2})$.
\end{proof}

From Lemma \ref{boundedness}, Condition \ref{con1} implies the boundedness of problem (P) and thus the optimal solution is always attainable \cite{jiang2017socp}.
 In the following theorem, we show how to recover the optimal solution of problem (P) from an optimal solution of problem $\rm(P_1)$.
\begin{thm}\label{recover}
 Assume that Condition \ref{con1} holds and $x^*$ is an optimal solution of problem $\rm(P_1)$. Then an optimal solution of problem $\rm(P)$ can be obtained in the following ways:
\begin{enumerate}
\item If $h_1(x^*)= t$ and $h_2(x^*)\leq t$, then $x^*$ is an optimal solution to $\rm(P)$;
\item  Otherwise $h_1(x^*)<t$ and $h_2(x^*)= t$.
For any vector $v_l\in \nl(A_2)$, let $\tilde \theta$ be a solution of the following equation,
 \begin{equation}\label{theta}
 h_1(x^*+\theta v_l)=\frac{1}{2}v_l^\top A_1v_{l}\theta^2 +(v_l^\top A_1x^{*  }+a_{1}^\top v_{l})\theta+h_1(x^*)=t.
 \end{equation}
Then $\{\tilde x: \tilde x=x^*+\tilde\theta v_{l},v_l\in\nl(A_2), \tilde \theta \text{ is a solution of (\ref{theta})}\}$ forms the set of optimal solutions of  $\rm(P)$.
\end{enumerate}
\end{thm}
\begin{proof}
Note that at least one of $h_1(x^*)\leq t$ and $h_2(x^*)\leq t$ takes equality. Then we prove the theorem for the following two cases:

\begin{enumerate}
\item
If $h_1(x^*)= t$ and $h_2(x^*)\leq t$,  then $f_1(x^*)+\lambda_2f_2(x^*)\leq f_1(x^*)+\lambda_1f_2(x^*)$. Hence $f_2(x^*)\leq0$ due to $\lambda_2-\lambda_1>0$.

\item Otherwise, $h_1(x^*)<t$ and $h_2(x^*)=t$. In this case, for all $d\in\nl(A_2)$ we have $d^\top A_1d>0$ due to Lemma \ref{SDinterval}. We also claim that $a_{2}^\top d=0$. Otherwise, setting $d$ such that $a_{2}^\top d<0$  (This can be done since we have $a_{2}^\top(-d)<0$ if $a_{2}^\top d>0$.) yields
$$h_2(x^{*}+d)=h_2(x^*)+\frac{1}{2}d^\top A_2d+(x^*)^\top A_2 d+a_{2}^\top d=h_2(x^*)+a_{2}^\top d<t,$$
where the second equality is due to $d\in\nl(A_2)$ and $h_{1}(x^{*}+d)<t$ for any sufficiently small $d$. This implies that $(x^{*},t)$ is not optimal, which is a contradiction.
Equation (\ref{theta}) has two solutions due to the positive parameter before the quadratic term, i.e., $v_l^\top A_1v_{l}>0$ and the negative constant, i.e., $h_1(x^*)-t<0$.
With the definition of $\tilde\theta$,
we know  $h_1(\tilde x)=t$ and $h_2(\tilde x)=t$. This further implies  $f_1(\tilde x)=t$ and $f_2(\tilde x)=0$, i.e., $\tilde x$ is an optimal solution to (P).
\end{enumerate}\end{proof}

\begin{rem}
In Item  2 of the above proof, $a_{2}^\top d=0$ indicates that problem (P) is in the hard case.\end{rem}


We next illustrate our recovery approach for the following simple     example,
$$\min\{3x_1^2-\frac{1}{2}x_2^2-x_2:-x_1^2+\frac{1}{2}x_2^2+x_2+1\leq0\}.$$
Note that, for this example, Condition \ref{con1} holds, $\lambda_1=1$ and $\lambda_2=3$. Then we have the following CQR,
$$\min\{t:2x_1^2+1\leq t,~x_2^2+2x_2+3\leq t\}.$$
An optimal solution of the CQR  is $x=(0,-1)^\top ,~t=2$. However, this $x$ is  not feasible to (P).
Using the approach in Theorem \ref{recover}, we obtain an optimal solution $\tilde x=(\frac{\sqrt2}{2},-1)^\top$   to problem (P). In fact, this instance is in the hard case since the optimal Lagrangian multiplier, $\lambda^*=3$, is at the end of the interval $\{\lambda:Q_1+\lambda Q_2\succeq0,~\lambda\geq0\}$ and $a-\lambda^*b\in\ra(Q_1+\lambda^* Q_2)$.

We finally point out that our method can be extended to the following variants of GTRS with equality constraint and interval constraint,
\begin{eqnarray*}
\begin{array}{llllll}
\rm(EP) &\min& f_1(x):=\frac{1}{2}x^\top Q_1x+b_{1}^\top x&\rm(IP) &\min& f_1(x):=\frac{1}{2}x^\top Q_1x+b_{1}^\top x\\
&{\rm s.t.}&f_2(x):=\frac{1}{2}x^\top Q_2x+b_{2}^\top x+c=0,&&{\rm s.t}.&c_1\leq f_2(x):=\frac{1}{2}x^\top Q_2x+b_{2}^\top x\leq c_2.
\end{array}
\end{eqnarray*}   It is shown in \cite{pong2014generalized,jiang2017socp} that \rm(IP) can be reduced to  \rm(EP) with minor computation. It is obvious that all our previous results for inequality constrained GTRS hold for (EP) if we   remove the non-negativity requirement for $\lambda$ in $I_{PSD}$, i.e., $I_{PSD}=\{\lambda:Q_1+\lambda Q_2\succeq0\}$.
We thus omit detailed discussion for (EP)  to save space.

In the last part of this section, we compare the   CQR in this paper with CQR for general QCQP in \cite{fujie1997semidefinite}.  The authors in \cite{fujie1997semidefinite} considered the following general QCQP,
$${\rm(QP)}~\min \tilde b_0^\top x ~~~{\rm s.t.} ~\frac{1}{2}x^\top \tilde Q _ix+\tilde b_{i}x+\tilde c_i\leq0,~i=1,\ldots,m,~x\in X,$$
where $X$ is a polyhedron.
They further showed that the SDP\ relaxation of (QP) is equivalent to  the following CQR for {\rm(QP)}:
$${\rm (CQP)}~~\min \tilde b_0^\top x ~~~{\rm s.t.} ~x\in G,$$
where $G=\{x:F_{s}(x)\leq0$ for every $s\in T$\},  $F_{s}(x)=\sum_{i=1}^ms_i(\frac{1}{2}x^\top \tilde Q_ix+\tilde b_{i}x+\tilde c_i)$ and
$$T:=\{s\in\mathbb{R}^m:~s\geq0,\tau\in \mathbb{R},  \begin{pmatrix}\sum_{i=1}^ms_{i}\tilde c_{i}&\frac{1}{2}(\sum_{i=1}^ms_{i}\tilde b_i)  \\\frac{1}{2}(\sum_{i=1}^ms_i\tilde b_i^\top)&\sum_{i=1}^m\frac{s_i}{2}\tilde Q_i\end{pmatrix}\succeq0\}.$$ For the  quadratic problem $\rm(P_1)$, because the variable $t$ is linear in the objective and the constraints, we can reduce $T$  to
$$T:=\{s:s_{1}+s_2=1,s\geq0,\sum_{i=1}^2\frac{s_i}{2}Q_i\succeq0,\sum_{i=1}^2s_{i}b_i\in\ra(\sum_{i=1}^2s_iQ_i)\},$$
where the restriction $s_1+s_2=1$ does not affect the feasible region $G$ since $F_s(x)\leq0$ is equivalent to  $k F_s(x)\leq0$ with any positive scaling $k$ for $s$.  Note that   $h_1(x)=F_{s^{1}}(x)$ and $h_2(x)=F_{s^{2}}(x)$ with $s^{1}=(1,0)^\top$ and $s^{2}=(0,1)^\top$. For any $s\in T$,   $h_1(x)\leq0$ and $h_2(x)\leq0$ imply $F_{s}(x)\leq0$ because
$F_{s}(x)$ is a convex combination of $f_{1}(x)+\lambda_1f_2(x)$ and $f_{1}(x)+\lambda_2f_2(x)$. Hence, by the strong duality and with analogous proof to that in Theorem \ref{mainthm}, the two  feasible regions  of problems $\rm(P_1)$ and $\rm(CQP)$ are equivalent and we further have $v{\rm(P_1)}= v{\rm(CQP)}$.

\section{Efficient algorithms in solving the minimax problem reformulation of the CQR }

In this section, we propose efficient algorithms to solve the GTRS under Condition \ref{con1}. As shown in Theorem \ref{mainthm} and Corollary \ref{cqr2},  the GTRS is equivalent to $\rm(P_1)$ or either   ${\rm(\overline P_1)}$ or ${\rm(\underline P_1)}$. The three problems have similar forms and can be solved by the following proposed method in this section. Hence, to save space, we only consider  solution algorithms for $\rm(P_1)$ in this section.

The convex quadratic problem $\rm(P_1)$  can be cast as an SOCP problem and solved by many existing solvers, e.g.,  CVX \cite{cvx}, CPLEX \cite{cplex} and MOSEK \cite{mosek2017mosek}. However, the SOCP reformulation is not very efficient when the dimension is large (e.g., the SOCP solver will take about 1,000 seconds to solve a problem of dimension 10,000).
Fortunately, due to its simple  structure, $\rm (P_1)$ is equivalent to the following  minimax problem of two convex quadratic functions$${\rm(M)}~~~\min \{H(x):=\max \{h_1(x),h_2(x)\}\}.$$
Hence we aim to derive an efficient method to solve the above minimax problem, thus solving the original GTRS.  Our method is a steepest descent method to find a critical point with $0\in\partial H(x)$. It is well known that such a critical point is an optimal solution of problem (M).

The following theorem tells us how to find the steepest descent direction.
\begin{thm}\label{descentthm}Let $g_1=\nabla h_1(x)$ and $g_2=\nabla h_2(x)$. If $g_1$ and $g_2$ have opposite directions, i.e., $g_1=-tg_2$ for some constant $t>0$ or if $g_i=0$ and $h_i(x)\geq h_j(x)$ for $i\neq j,$ $i,j\in\{1,2\}$, then $x$ is a global optimal solution. Otherwise we can always find the steepest descent direction $d$ in the following way:
\begin{enumerate}
\item when  $h_1(x)\neq h_2(x)$,  $d=-g_1$ if $h_1(x)>h_2(x)$ and otherwise $d=-g_2$;
\item when $h_1(x)=h_2(x)$, $d=-(\alpha g_1+(1-\alpha)g_2)$, where $\alpha$ is defined in the following three cases:
\begin{enumerate}
\item $\alpha=0$, if  $g_1^\top g_1\geq g_1^\top g_2\geq g_2^\top g_2$,
\item $\alpha=1$, if $g_1^\top g_1\leq g_1^\top g_2\leq g_2^\top g_2$,\item $\alpha=\frac{g_2^\top g_2-g_1^\top g_2}{g_1^\top g_1+g_2^\top g_2-2g_1^\top g_2}$, if $g_1^\top g_2\leq g_2^\top g_2$ and  $g_1^\top g_2\leq g_1^\top g_1$.
\end{enumerate}
\end{enumerate}
\end{thm}
\begin{proof}
If
 $h_1(x)=h_2(x)$ and $g_1=-tg_{2}$, then  $0\in \partial H(x)$. Hence, by the definition of subgradient, we have
 $$H(y)\geq H(x)+0^\top (y-x)=H(x),~ \forall y,$$which further implies that $x$ is the optimal solution.

If $g_i=0$ and $h_i(x)\geq h_j(x)$ for $i\neq j,$ $i,j\in\{1,2\}$, then for all $y\neq x,$ we have $H(y)\geq h_i(y)\geq h_i(x)=H(x)$, i.e.,  $x$ is a global optimal solution.

Otherwise we have the following three cases:
\begin{enumerate}
\item When  $h_1(x)\neq h_2(x)$, (suppose, w.l.o.g., $h_1(x)>h_2(x)$),  for all $y\in\mathcal{B}(x,\delta)$ with $\mathcal{B}(x,\delta)\subset\ \{x:h_2(x)<h_1(x)$\}), $H(x)=h_1(x)$ and thus $H(x)$ is differentiable at $x$ and smooth in its
neighbourhood. Hence, $d=-g_1$ if $h_1(x)>h_2(x)$. Symmetrically, the case with $h_2(x)>h_1(x)$ can be proved in the same way.
\item When $h_1(x)=h_2(x)$, the steepest descent direction can be found by solving the following problem:
$$\min_{\norm{y}=1}\max_{g\in\partial H(x)}g^Ty.$$
The above  problem is equivalent to $\min_{g\in\partial H(x)}\norm{g}^2$ \cite{boyd2004convex}, which is
exactly the following problem in minimizing a quadratic function of $\alpha$,
\begin{equation}\label{std}
\min_{0\leq \alpha\leq1} (\alpha g_1+(1-\alpha)g_2)^\top(\alpha g_1+(1-\alpha)g_2). \end{equation}
The first order derivative  of the above objective function is $\frac{g_2^\top g_2-g_1^\top g_2}{g_1^\top g_1+g_2^\top g_2-2g_1^\top g_2}$. Then if the derivative is in the interval $[0,1]$, the optimal $\alpha$ is given by $\frac{g_2^\top g_2-g_1^\top g_2}{g_1^\top g_1+g_2^\top g_2-2g_1^\top g_2}$. Otherwise,
(\ref{std}) takes its optimal solution on its boundary. In particular, \begin{itemize}
\item when
  $\frac{g_2^\top g_2-g_1^\top g_2}{g_1^\top g_1+g_2^\top g_2-2g_1^\top g_2}>1$, i.e., $g_1^\top g_1<g_1^\top g_2$ and $g_2^\top g_2>g_1^\top g_2$,  we have $\alpha=1$,
\item when
 $\frac{g_2^\top g_2-g_1^\top g_2}{g_1^\top g_1+g_2^\top g_2-2g_1^\top g_2}<0$, i.e., $g_1^\top g_2>g_2^\top g_2$,   we have $\alpha=0.$
\end{itemize}
\end{enumerate}
\end{proof}
\begin{rem}\label{smallestnorm}
The above theorem shows that the descent direction at each point with $h_1(x)=h_2(x)$ is either the one with the smaller norm between $\nabla h_1(x)$ and $\nabla h_2(x)$
 or the negative convex combination $d$ of  $\nabla h_1(x)$ and $\nabla h_2(x)$
 such that $\nabla h_1(x)^\top d=\nabla h_1(x)^\top d $. \end{rem}
 \begin{figure}[ht]
\centering
\includegraphics[width=\textwidth]{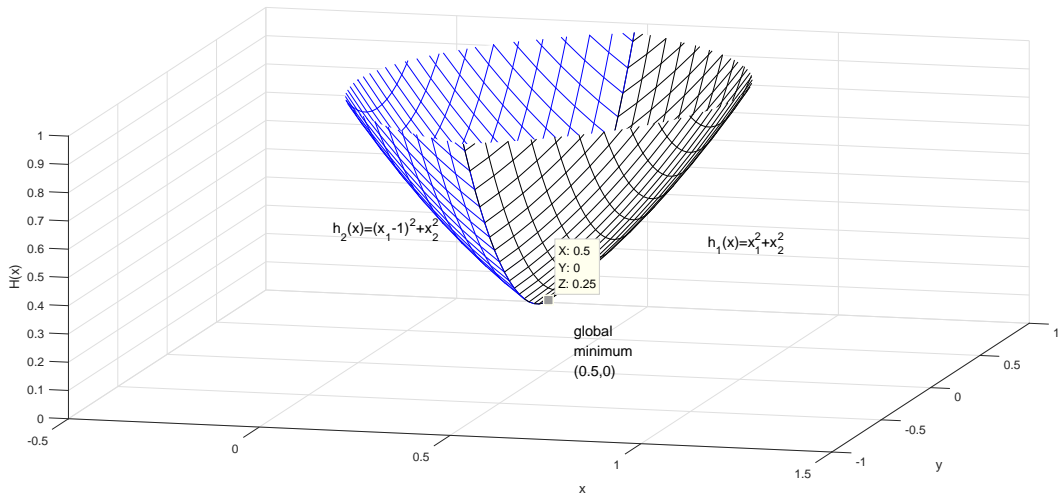}
\caption{Involving subgradient can avoid termination at nonsmooth and non-optimal points.}\label{fig1}
\end{figure}
We next present an example in Figure \ref{fig1} to illustrate the necessity of involving the subgradient (in some cases, both gradients are not descent directions). Consider $h_1(x)=x_1^2+x_2^2$ and $h_2(x)=(x_1-1)^2+x_2^2$. The optimal solution of this problem is $(0.5,0)^\top$. The gradient method can only converge to some point in the intersection curve of $h_1(x)=h_2(x)$, i.e., $x_{1}=0.5$, but not the global optimal solution. For example, when we are at $\bar x=(0.5,0.1)^{\top}$, the gradients for $h_1(\bar x)$ and $h_2(\bar x)$ are $g_1=(1,0.2)^\top$ and $g_2=(-1,0.2)^\top$, respectively. Neither $-g_1$ nor $-g_2$ is a descent direction at $H(\bar x)$; $H(\bar x+\epsilon g_i)>H(\bar x)$ for any small $\epsilon>0$, $i=1,2$, due to $g_{1}^\top g_2=-0.96<0$ and $h_1(\bar x)=h_2(\bar x)$. (The direction $-g_1$ is a descent direction, at $\bar x$, for $h_1(x)$  but ascent for $h_2(x)$ and thus ascent for $H(x)$; the same analysis holds for $-g_2$.)
The way we use to conquer this difficulty is to choose the steepest descent direction in the subgradient set at  points in the intersection curve. If we use the subgradient direction $d=-\frac{1}{2}(g_1+g_2)=-(0,0.2)^\top$, then $d$ is a descent direction since
$h_{1}(\bar x+\epsilon d)=H(\bar x)+2\epsilon g_{1}^\top d+\epsilon^2 d^\top d<H(\bar x)$ and $h_{2}(\bar x+\epsilon d)=H(\bar x)+2\epsilon g_{1}^\top d+\epsilon^2 d^\top d<H(\bar x)$  for any $\epsilon$ with $0<\epsilon<2$.

Using the descent direction presented in Theorem \ref{descentthm}, we propose two  algorithms to solve the minimax problem (M), respectively, in Algorithms \ref{alg1} and \ref{alg2}:
 we first compute a descent direction by Theorem \ref{descentthm}, apply then two different line search rules for choosing the step size, and finally terminate the algorithm if some termination criterion is met.   The advantage of our algorithms is that each iteration is  very cheap, thus yielding, with an acceptable iterate number, a low cost in CPU time. The most expensive operation  in each iteration is to compute several matrix vector products, which could become cheap  when the matrices are sparse.

\subsection{Line search with a special step size}
\begin{algorithm}[!ht]
\caption{Line search with a special step size for Problem (M)}
\label{alg1}
\begin{algorithmic}[1]
\Require Parameters in the minimax problem (M) \State Initialize  $x_0$\For {$k=0,1,\ldots, $}
\If {$h_1(x_k)>h_2(x_k)$} \State set $d_{k}=-\nabla h_{1}(x_k)$
\ElsIf {$h_1(x_k)<h_2(x_k)$}
\State set $d_{k}=-\nabla h_{2}(x_k)$
\Else \State set $d_{k}$ corresponding to Theorem \ref{descentthm}, item 2
\EndIf
\If {termination criterion is met}
\Return
\EndIf
\State Choose a step size $\beta_k$ according to Theorem \ref{lincon}
\State Update $x_{k+1}=x_k+\beta_kd_{k}$
\EndFor
\end{algorithmic}
\end{algorithm}

In the following, we first derive a local linear convergence rate for Algorithm \ref{alg1} and then demonstrate a  global sublinear convergence rate for Algorithm \ref{alg1}. We analyze the local convergence rate by studying the growth in the neighbourhood of any optimal solution to $H(x)$ in (M). In fact, $H(x)$ belongs to a more general class of piecewise quadratic functions. Error bound and KL property, which are two widely used techniques for convergence analysis, have been studied in the literature,  for several kinds of piecewise quadratic functions, see \cite{li1995error, luo2000error,zhou2017unified}. However, these results are based on piecewise quadratic functions separated by polyhedral sets, which is not the case of $H(x)$.
Li et al. \cite{li2015new} demonstrated  that KL property holds for the maximum of finite polynomials, but their KL exponent depends on the problem dimension and is close to one, which  leads to a very weak sublinear convergence rate.
Gao et al. \cite{gao2016ojasiewicz} studied the KL exponent for the TRS with the constraint replaced by an equality constraint $x^\top x=1$.
However, their technique depends on the convexity of the function $x^\top x$ and cannot be applied to analyze our problem. Up to now,  the KL property or error     bound for $H(x)$ has not been yet investigated in the literature related to the linear convergence of optimization algorithms. A significant result of this paper is to estimate the KL exponent of $1/2$ for function $H(x)$ when $\min_{x} H(x)> \max_{i}\{\min_{x} h_1(x),\min_{x} h_2(x)\}$. With this KL  exponent, we are able to illustrate the linear convergence of our first algorithm with the proposed special step size.

For completeness, we give a definition of KL property in the following. By letting $\mathcal{B}(x,\delta)=\{y:\norm{y-x}\leq \delta\}$  where $\norm{\cdot}$ denotes the Euclidean norm of a vector, we have the following definition of KL inequality.
\begin{dfn}\label{dfnkl}\cite{attouch2009convergence,gao2016ojasiewicz}
Let $f:\mathbb R^n\rightarrow\mathbb R\cup\{+\infty\}$ be a proper lower semicontinuous function satisfying that the restriction of $f$
to its domain is a continuous function.  The function $f$ is said to have the KL property if for any  $\forall x^*\in\{x:0\in\partial f(x)\}$,
there exist $C,\epsilon>0$ and $\theta\in[0,1)$ such that \begin{equation*}\label{KLiq}
 C\norm{y}\geq|f(x)-f^*(x)|^\theta,~~~\forall x\in B(x^*,\epsilon),~\forall y\in\partial f(x),
\end{equation*}
where $\theta$ is known as the KL exponent.

\end{dfn}

Under Condition \ref{con1}, we know that there exists $\lambda_0\geq0$ such that $Q_1+\lambda_0Q_2\succ0$ and thus  $b_1+\lambda_0b_2\in \ra( Q_1+\lambda_0Q_2)$ due to the non-singularity of $Q_1+\lambda_0Q_2$. Hence from Lemma \ref{nscon}, problem (P) (and thus problem $\rm(P_1)$) is bounded from below. It is shown in \cite{jiang2017socp} that when the two matrices are SD and problem (P) is bounded from below,  the optimal solution of problem (P) is attainable. This further implies that problem $\rm(P_1)$ is  bounded from below with its optimal solution attainable. Assuming that $x^*$ is an optimal solution, the following theorem shows that the KL inequality holds with an exponent of 1/2  at $x^*$ under some mild conditions.
\begin{thm}\label{KL}
Assume that $\min h_1(x)<\min H(x)$ and  $\min h_2(x)<\min H(x)$. Then the KL property in Definition \ref{dfnkl} holds with exponent $\theta=1/2$. \end{thm}
\begin{proof}
Note that $\min h_1(x)<\min H(x)$ and  $\min h_2(x)<\min H(x)$ imply that, for any $x^*\in\{x:\partial H(x)=0\}$, $\nabla h_1(x^*)\neq0$ and $\nabla h_2(x^*)\neq0$, respectively.
Assume $L= \max \{\lambda_{\max}(A_1),\lambda_{ \max}(A_2)\}.$ We carry out our proof by considering the following two cases.
\begin{enumerate}
\item For any point with $h_1(x)\neq h_2(x)$, w.l.o.g., assuming $h_1(x)>h_2(x)$ gives rise to  $\partial H(x)=\{\nabla h_1(x)\}$. Hence
\begin{eqnarray*}
|H(x)-H(x^*)|&=&\frac{1}{2}(x-x^*)^\top A_1(x-x^*)+(x^*)^\top A_1(x-x^*)+a_1^\top(x-x^*)\\
&\leq& \frac{1}{2}L\norm{x-x^*}^2+\norm{\nabla h_{1}(x^*)}\norm{x-x^*}.
\end{eqnarray*}

On the other hand, $\nabla h_1(x)=A_1x+a_1$ and
\begin{eqnarray*}
\norm{\nabla h_1(x)}^2&=&\norm{\nabla h_1(x)-\nabla h_1(x^*)+\nabla h_1(x^*)}^2\\
&=&(x-x^*)^\top A_1A_1(x-x^*)+\norm{\nabla h_1(x^*)}^2+2(\nabla h_1(x^*))^\top A_1(x-x^*) \\
&\geq&\norm{\nabla h_1(x^*)}^2-2L\norm{\nabla h_1(x^*)}\norm{x-x^*}.
\end{eqnarray*}
Define  $\epsilon_{0}=\min\{1,\frac{\norm{\nabla h_1(x^*)}}{4L}\}$. As $\nabla h_1(x^*)\neq0$, for all $x\in \mathcal{B}(x^*,\epsilon_{0})$,  we then have
$$|H(x)-H(x^*)|\leq \frac{1}{2}L\epsilon_{0}^2+ \norm{\nabla h_{1}(x^*)}\epsilon_{0}\leq \frac{9}{32L}\norm{\nabla h_{1}(x^*)}^2$$
and
\begin{eqnarray*}
\norm{\nabla h_1(x)}^2\geq\norm{\nabla h_1(x^*)}^2-2L\norm{\nabla h_1(x^*)}\epsilon_{0}\geq\frac{1}{2}\norm{\nabla h_1(x^*)}^2.
\end{eqnarray*}
Hence ${|H(x)-H(x^*)|}^{\frac{1}{2}}\leq \sqrt{\frac{9}{32L}}\norm{\nabla h_{1}(x^*)}\leq \frac{3}{4\sqrt{L}}\norm{\nabla h_1(x)}$. So we have the following
 inequality,
$$|H(x)-H(x^*)|^\theta\leq C_{0}\norm{y},$$ for all $y\in\partial H(x)$ (here $\{\nabla h_1(x)\}=\partial H(x)$) with $\theta=\frac{1}{2}$, $C_{0}=\frac{3}{4\sqrt{L}}$.

\item Consider next a point $x$ with $h_1(x)=h_2(x)$. Define $h_\alpha(x)=\alpha h_1(x)+(1-\alpha)h_2(x),$ for some parameter $\alpha\in[0,1]$. Let $I=\{i\mid(\nabla h_1(x^*))_i\neq0\}$.  The optimality condition $0\in\partial H(x^*)$ implies that there exists some $\alpha_0\in[0,1]$ such that $\alpha_0\nabla h_1(x^*)+(1-\alpha_0)\nabla h_2(x^*)=0$.
 Note that  $\nabla h_1(x^*)\neq0$ and $\nabla h_2(x^*)\neq0$ as assumed and thus $\alpha_0\in(0,1)$. Define $j={\rm argmax}_i\{|(\nabla h_1(x^*))_i|,i\in I\},~M_{1}=\norm{(\nabla h_1(x^*))_j}$ and $M_2=\norm{(\nabla h_2(x^*))_j}.$
Note that $\partial H_{\alpha_0}(x^{*})=0$ implies that $\alpha_0 M_1=(1-\alpha_0) M_2$.  W.o.l.g, assume  $M_1\geq M_2$ and thus $\alpha_0\leq \frac{1}{2}$.
Since $A_1x$ ($A_2x$, respectively) is a continuous function of $x$, there exists an $\epsilon_1>0$ ($\epsilon_2>0$, respectively) such that for any $x\in \mathcal{B}(x^*,\epsilon_1)$ ($x\in\mathcal{B}(x^*,\epsilon_2)$, respectively), $\frac{3}{2}M_1\geq|(\nabla h_1(x))_j| >\frac{1}{2}M_1$ ($\frac{3}{2}M_2\geq|(\nabla h_2(x))_j| >\frac{1}{2}M_2$, respectively).
Let $\epsilon_{3}=\min \{\epsilon_1,\epsilon_2\}$. Then we have the following two subcases.
\begin{enumerate}
\item
For all $x\in \mathcal{B}(x^*,\epsilon_{3})$ and $\alpha\in[0,\frac{1}{4}\alpha_0]$, we  have
\begin{eqnarray*}
\norm{\nabla h_\alpha(x)}&\geq&-\alpha |(\nabla h_1(x))_j|+(1-\alpha) |(\nabla h_2(x))_j|\\
&\geq&-\frac{3}{2}\alpha M_1+\frac{1}{2}(1-\alpha)M_2\\
&\geq&-\frac{3}{8}\alpha_0 M_1+\frac{3}{8}(1-\alpha_0)M_2+(\frac{1}{8}+\frac{1}{8}\alpha_0)M_2\\
&=&(\frac{1}{8}+\frac{1}{8}\alpha_0)M_2.
\end{eqnarray*}
The third inequality is due to the fact that $-\frac{3}{2}\alpha M_1+\frac{1}{2}(1-\alpha)M$ is a decreasing function of $\alpha$ and the last equality is due to $\alpha_0 M_1=(1-\alpha_0) M_2$.
Symmetrically,  for $\alpha\in[1-\frac{1-\alpha_0}{4},1]$, we have $|(\nabla h_\alpha(x))|\geq(\frac{3}{8}-\frac{1}{4}\alpha_0)M_1$.
Combining these two cases and  $\alpha_0\leq \frac{1}{2}$
yields $\norm{\nabla h_\alpha(x)}\geq\frac{1}{8}M_2.$

On the other hand
\begin{eqnarray*}
|H(x)-H(x^*)|&=&\frac{1}{2}(x-x^*)^\top A_1(x-x^*)+(x^*)^\top A_1(x-x^*)+a_1^\top(x-x^*)\\
&\leq&\frac{1}{2}L\norm{x-x^*}^2+\norm{\nabla h_{1}(x^*)}\norm{x-x^*}\\
&\leq&\left(\frac{1}{2}L\epsilon_{3}^2+\norm{\nabla h_{1}(x^*)}\right)\norm{x-x^*}.
\end{eqnarray*}
Letting $\epsilon_{4}=\min \{\epsilon_{3},1\}$ leads to $\frac{M_2^2}{32L\epsilon_{3}^{2}+64\norm{\nabla h_{1}(x^*)}}|H(x)-H(x^*)|\leq \norm{\nabla h_\alpha(x)}^{2}$. So $$|H(x)-H(x^*)|^\theta\leq C_{1}\norm{\nabla h_\alpha(x)},~\forall \alpha\in[0,\frac{1}{4}\alpha_0]\cup[1-\frac{1-\alpha_0}{4},1],~\forall x\in \mathcal{B}(x^*,\epsilon_{4})$$where $\theta=\frac{1}{2}$ and $C_{1}=\frac{\sqrt {32L\epsilon_{3}^{2}+64\norm{\nabla h_{1}(x^*)}}}{M_2}$.

\item Next let us consider the case with $\alpha\in[\frac{\alpha_0}{4},1-\frac{1-\alpha_0}{4}]$.
In this case,
 defining $A_\alpha=\alpha A_1+(1-\alpha)A_2$
and $a_\alpha=\alpha a_1+(1-\alpha)a_2$ gives rise to
\begin{eqnarray*}
\norm{\nabla h_\alpha(x)}^2&=&\norm{\nabla h_\alpha(x)-\nabla h_\alpha(x^*)+\nabla h_\alpha(x^*)}^2\\
&=&(x-x^*)^\top A_\alpha A_\alpha(x-x^*)+\norm{\nabla h_\alpha(x^*)}^2+2(\nabla h_\alpha(x^*))^\top A_\alpha(x-x^*)
\end{eqnarray*}
and since $h_1(x)=h_2(x)$ and $h_1(x^*)=h_2(x^*)$,
\begin{eqnarray*}
|H(x)-H(x^*)|&=&\frac{1}{2}(x-x^*)^\top A_\alpha(x-x^*)+(x^*)^\top A_\alpha(x-x^*)+a_\alpha^\top(x-x^*)\\
&=&\frac{1}{2}(x-x^*)^\top A_\alpha(x-x^*)+(\nabla h_\alpha(x^*))^\top(x-x^*).
\end{eqnarray*}
Define $\mu_{0}=\lambda_{\min}(A_\alpha)$. Then
\begin{eqnarray*}
&&\norm{\nabla h_\alpha(x)}^2-2\mu_{0}|H(x)-H(x^*)|\\
&=&(x-x^*)^\top A_\alpha (A_\alpha-\mu_{0} I)(x-x^*)+\norm{\nabla h_\alpha(x^*)}^2\\
&&+2(\nabla h_\alpha(x^*))^\top(A_\alpha-\mu_{0} I)(x-x^*) \\
&=&\norm{(A_\alpha-\mu_{0} I)(x-x^*)+\nabla h_\alpha(x^*)}^2+\mu_{0}(x-x^*)^\top(A_\alpha-\mu_{0} I)(x-x^*)\geq0.
\end{eqnarray*}

We next show that $\mu_{0}$ is  bounded from below. Define $\mu_1$ ($\mu_2$, respectively) as the smallest nonzero eigenvalue of $A_1$ ($A_2$, respectively).
Note that $\alpha A_1+(1-\alpha)A_2$ is positive definite for all $\alpha\in[\frac{\alpha_0}{4},1-\frac{1-\alpha_0}{4}]$ as assumed in Condition \ref{con1}. Then $A_1$ and $A_2$ are simultaneously diagonalizable as shown in Lemma \ref{SDinterval}. Together with the facts that $A_1\succeq0$ and $A_{2}\succeq0$, there exists a nonsingular matrix $P$ such that $P^\top A_1P=D_1\succeq \mu_1\dg(\delta )$ and $P^\top A_2P=D_2\succeq \mu_2\dg(\delta )$, where $\delta_i=1$ if $D_{ii}>0$ and $\delta_i=0$ otherwise.
Since $\alpha\in[\frac{\alpha_0}{4},1-\frac{1-\alpha_0}{4}]$, $\lambda_{\min}(A_\alpha)\geq \min\{\alpha\mu_{1}, ㄗ(1-\alpha)\mu_2\}\geq\min\{ \frac{\alpha_0}{4}\mu_{1},\frac{1-\alpha_0}{4}\mu_2\}>0$.
From  $\norm{\nabla h_\alpha}^2-2\mu_{0}|H(x)-H(x^*)|\geq0$, we know $\norm{\nabla h_\alpha}^2-\mu_{0}|H(x)-H(x^*)|\geq0$.

Let $\theta=\frac{1}{2}$, $C_{2}=\sqrt{1/(2\mu_{0})}$. We have
$$C_{2}\norm{\nabla h_\alpha(x)}\geq|H(x)-H(x^*)|^{\theta},
~\forall\alpha\in[\frac{\alpha_0}{4},1-\frac{1-\alpha_0}{4}],~x\in \mathcal{B}(x^*,\epsilon_{4}).$$
\end{enumerate}
Combining cases (a) and (b) gives rise to $$|H(x)-H(x^*)|^\theta\leq C_{3}\norm{\nabla h_\alpha(x)}$$ with $\theta=\frac{1}{2}$, $C_{3}=\max\{C_{1},C_{2}\}$, for all $x\in \mathcal{B}(x^*,\epsilon_{4})$.
\end{enumerate}
Combining cases 1 and 2 yields that the KL inequality holds  with $\theta=\frac{1}{2}$ and $C=\max\{C_{0},C_{3}\}$   for all $x\in \mathcal{B}(x^*,\epsilon)$ with $\epsilon=\min\{\epsilon_0,\epsilon_4\}$.
\end{proof}

Note that the assumption  $\min h_1(x)<\min H(x)$ and  $\min h_2(x)<\min H(x)$ means that we are in the easy case of GTRS as in this case $\lambda^*$ is an interior point of $I_{PSD}$ and  $Q_1+\lambda^* Q_2$ is nonsingular, where $\lambda^*$  is the optimal Lagrangian multiplier of the GTRS \cite{more1993generalizations}. However, there are two situations for the hard case. Let us consider the KL property for $H(x)$ at the optimal solution $x^*$. When $h_i(x^{*})>h_j(x^{*})$, for $i=1$ or 2 and $j=\{1,2\}/\{i\}$, in the neighbourhood $x^*$,   $H(x)$ is just $h_i(x)$, and the KL is also 1/2 \cite{attouch2009convergence}. In such a case, our algorithm performs asymptotically like the gradient descent method for unconstrained quadratic minimization.
However, when  $h_i(x^{*})=h_j(x^{*})$ (note that $\min h_j(x)<H(x^{*})$ can still hold in this situation),
 the KL exponent is not always $1/2$ for $H(x)$.  Consider the following counterexample with
$h_1(x)=x_1^2$ and $h_2(x)=(x_1+1)^2+x_2^2-1$. The optimal solution is $(0,0)$ and is attained by both $h_1$ and $h_2$.
Let $x_2=-\epsilon$, where $\epsilon$ is a small positive number. Consider the curve where $h_1(x)=h_2(x)$, which further implies $x_1=-\epsilon^2/2$.
Then we have
$$(1-\beta ) \nabla h_1+\beta \nabla h_2=2\begin{pmatrix}-(1-\beta) \frac{\epsilon^2}{2}+\beta(-\frac{\epsilon^2}{2}+1)\\
\beta\epsilon
\end{pmatrix}=2\begin{pmatrix}- \frac{\epsilon^2}{2}+\beta\\
\beta\epsilon
\end{pmatrix},$$
and thus
\begin{eqnarray*}
\min_{y\in\partial H(x)}\norm{y}^2&=&\min_\beta4\left(\beta^2\epsilon^2+\beta^{2}-\epsilon^2\beta+\frac{\epsilon^4}{4}\right)\\
&=&\min_\beta4\left((1+\epsilon^2)\left(\beta-\frac{\epsilon^{2}}{2(1+\epsilon^{2})}\right)^2-\frac{\epsilon^{4}}{4(1+\epsilon^{2})}+\frac{\epsilon^4}{4}\right)\\
&=&\frac{\epsilon^{6}}{2(1+\epsilon^{2})}=\frac{\epsilon^6}{2}+O(\epsilon^8).\end{eqnarray*}
Thus, $\min_{y\in\partial H(x)}\norm{y}= O(\epsilon^3)$. On the other hand,
$$H(x)-H(x^*)=x_1^2=\frac{\epsilon^4}{4}.$$
The KL inequality cannot hold\ with $\theta=1/2$, but it holds with  $\theta=3/{4}$ since $\min_{y\in\partial H(x)}\norm{y}= O(\epsilon^3)$ and
$H(x)-H(x^*)=O(\epsilon^4)$.

\begin{rem}It  is interesting to compare our result with a recent result on KL exponent of the quadratic sphere constrained optimization problem \cite{gao2016ojasiewicz}. In \cite{gao2016ojasiewicz}, the authors showed that the KL exponent is $3/4$ in general and $1/2$ in some special cases, for the following problem, $$(T)~~~\min \frac{1}{2}x^\top Ax+b^\top x~~~~{\rm s.t. }~x^\top x=1.$$
The above problem is equivalent to the TRS when the constraint of the TRS is active, which is the case of interest in the literature. For the TRS, the case that the constraint is inactive is trivial: Assuming $x^*$ being the optimal solution, $(x^*)^\top x^*<1$ if and only if the objective function is convex and the optimal solution of the unconstrained quadratic function $\frac{1}{2}x^\top Ax+b^\top x$ locates in the interior of the unit ball. The authors in   \cite{gao2016ojasiewicz} proved that the KL exponent is $3/{4}$ in general and particularly the KL exponent is $1/{2}$ if $A-\lambda^*I$ is nonsingular, where $\lambda^*$ is the optimal Lagrangian multiplier. The later case is a subcase of  the easy case  for the TRS and the case that KL exponent equals $3/{4}$ only occurs in some special situations of the hard case.  On the other hand, our result shows the KL exponent is  $1/2$  for the minimax problem when the associated GTRS is in the easy case.
So our result can be seen as an extension of the resents on KL exponent for  problem (T) in \cite{gao2016ojasiewicz}. One of our future research is to verify if the KL exponent is $3/{4}$ for $H(x)$ when the  associated GTRS is in the hard case.
\end{rem}

For convergence analysis with error bound or KL property, we still need a sufficient descent property to achieve the convergence rate. We next propose an algorithm with such a property. We further show that our algorithm converges locally linearly with the descent direction  chosen  in Theorem \ref{descentthm} and the step size specified  in the following theorem.

\begin{thm}\label{lincon}
Assume that the conditions in Theorem \ref{KL} hold and that the initial point $x^0\in \mathcal{B}(x^*,\epsilon)$.  Assume that $h_i$ is the active function when $h_1(x_k)\neq h_2(x_k)$ and $h_j$, $j=\{1,2\}/\{i\}$, is thus inactive. \textcolor{red}{} Let the descent direction   be chosen in Theorem \ref{descentthm} and the associated step size be chosen as follows.
\begin{enumerate}
\item When  $ h_1(x_k)= h_2(x_k)$,
\begin{itemize}
\item if  there exists   $g_\alpha=\alpha \nabla h_1(x_k)+(1-\alpha)\nabla h_2(x_k)$ with $\alpha\in[0,1]$ such that $\nabla h_1(x_k)^\top g_\alpha=\nabla h_2(x_k)^\top g_\alpha$, then set $d_k=-g_\alpha$ and $\beta_{k}=1/L$, where $L=\max \{\lambda_{max}(A_1),\lambda_{max}(A_2)\}$;
\item otherwise  set $d_k=-\nabla h_i(x_k)$ for $i$ such that  $\nabla h_1(x_k)^\top\nabla h_2(x_k)\geq  \nabla h_i(x_k)^\top \nabla h_i(x_k)$, $i=1,2,$ and $\beta_{k}=1/L$.

\end{itemize}
\item  When $h_1(x_k)\neq h_2(x_k)$ and the following quadratic equation for $\gamma$,
\begin{eqnarray}\label{qe}
\mathbf{a}x^2+\mathbf{b}x+\mathbf{c}=0,
\end{eqnarray}
where $\mathbf{a}=\frac{1}{2}\gamma^2\nabla h_i(x_k)^\top (A_i-A_j)\nabla h_i(x_k)$, $\mathbf{b}=(\nabla h_i(x_k)^\top -\nabla h_j(x_k)^\top )\nabla h_i(x_k)$ and $\mathbf{c}= h_i(x_k)-h_{j}(x_k)$,
 has no positive solution or any positive solution $\gamma\geq1/L$, set $d_k=-\nabla h_i(x_k)$ with  and $\beta_{k}=1/L$;
\item When $h_1(x_k)\neq h_2(x_k)$ and the quadratic equation (\ref{qe})  has a positive solution $\gamma<1/L$,  set $\beta_{k}=\gamma$  and $d_k=-\nabla h_i(x_k)$.
\end{enumerate}
 Then  the sequence $\{x_k\}$ generated by Algorithm \ref{alg1} satisfies, for any $k\geq1$,
\begin{eqnarray}\label{itebd}
H(x_k)-H(x^*)\leq \left(\sqrt\frac{2C^2L-1}{2C^2L}\right)^{k-1}(H(x^0)-H(x^*)),
\end{eqnarray}
and  $${\rm dist}(x_k,\overline X)^2\leq\frac{2}{L}(H(x_k)-H(x^{*})\leq\frac{2}{L}\left(\sqrt\frac{2C^2L-1}{2C^2L}\right)^{k-1}(H(x^0)-H(x^*)) .$$\end{thm}
\begin{proof}
For simplicity, let us denote $g_i=\nabla h_i(x_k)$  for $i=1,2$. We claim the following sufficient descent property for  steps 1,  2 and 3:
$$H(x_k)-H(x_{k+1})\geq \frac{L}{2}\norm{x_k-x_{k+1}}^2.$$
  Hence,
if the step size is $1/L$ (i.e., steps 1 and 2), we have$$ H(x_l)-H(x^{*})\leq C\norm{d_l}^{2}= C^2L^2\norm{x_l-x_{l+1}}^{2}\leq 2C^2L\left(H(x_l)-H(x_{l+1})\right),$$
where the first inequality is due to the KL inequality in Theorem \ref{KL}, the second equality is due to $x_{l+1}=x_l-\frac{1}{L}d_l$ and the last inequality is due to the sufficient descent property. Rearranging the above inequality yields
$$H(x_{l+1})-H(x^*)\leq \frac{2C^2L-1}{2C^2L}(H(x_l)-H(x^*)).$$
And since our method is a descent method, we have
$H(x_{l+1})-H(x^*)\leq H(x_l)-H(x^*)$ for all iterations. Suppose that there are $p$ iterates of step size 1, $q$ iterates of step size 2, and $r$ iterates of step size 3.
From the definitions of the steps, every step 3 is followed by  a step 1 and thus $r\leq p+1$ if we terminate our algorithm at step 1 or 2. So for all $k\geq1$, after $k=p+q+r$ steps, we have
$$H(x_k)-H(x^*)\leq\left(\frac{2C^2L-1}{2C^2L}\right)^{p+q}(H(x^{0})-H(x^*))\leq \left(\frac{2C^2L-1}{2C^2L}\right)^\frac{k-1}{2}(H(x^0)-H(x^*)).$$

The sufficient descent property further implies that $$\frac{L}{2}\sum_k^\infty\norm{x_k-x_{k+1}}^{2}\leq H(x_k)-H(x^{*}).$$ Hence, with $\sum_k^\infty\norm{x_k-x_{k+1}}^{2}\geq{\rm dist}(x_k,\overline X)^2$, we have $\frac{L}{2}{\rm dist}(x_k,\overline X)^{2}\leq H(x_k)-H(x^{*})$.
Thus  $${\rm dist}(x_k,\overline X)^{2}\leq\frac{2}{L}(H(x_k)-H(x^{*})).$$

By noting $g_i=A_ix_k+a_i$, we have
\begin{eqnarray*}
h_i(x_{k+1})-h_i(x_k)&= &\frac{1}{2}(x_k+d_k)^\top A_i(x_k+d_k)+a_i^\top(x_k+d_k)-[\frac{1}{2}(x_k)^\top A_ix_k+a_i^\top x_k]\\
&=&\frac{1}{2}d_k^\top A_id_{k}+(A_ix_k+a_i)^\top d_k\\
&=&\frac{1}{2}d_k^\top A_id_{k}+g_i^\top d_k.\end{eqnarray*}
We next prove our claim (\ref{itebd}) according to the three cases in our updating rule:\begin{enumerate}
\item
 When $h_1(x_k)= h_2(x_k)$, noting that $h_i$  is active at $x_{k+1}$ as assumed,  we have$$H(x_k)-H(x_{k+1})= h_i(x_k)-h_i(x_{k+1}).$$
\begin{itemize}
\item If there exists an $\alpha$ such that $g_\alpha^\top g_1=g_\alpha^\top g_2$, we have $g_\alpha^\top g_i=g_\alpha^\top g_\alpha$. And by noting that $d_i=-g_\alpha$, we further  have
\begin{eqnarray*}
h_i(x_{k+1})-h_i(x_k)&=&\frac{1}{2L^{2}}d_{k}^\top A_id_{k}+\frac{1}{L}g_i^\top d_{k}\notag\\
&\leq&\frac{1}{2L}g_\alpha^\top g_\alpha-\frac{1}{L}g_\alpha^\top g_\alpha\notag\\
&=&-\frac{1}{2L}g_\alpha^\top g_\alpha.
\end{eqnarray*}
Substituting $g_\alpha=L(x_k-x_{k+1})$ to the above expression, we have the following sufficient descent property,
\begin{equation*}
H(x_k)-H(x_{k+1})=h_i(x_k)-h_i(x_{k+1})\geq\frac{L}{2}\norm{x_k-x_{k+1}}^2.\label{suffdes}
\end{equation*}
\item
If there does not exist an $\alpha$ such that $g_\alpha^\top g_1=g_\alpha^\top g_2$, then we must have $g_1^\top g_2>0$. And thus we must have $g_1^\top g_1\geq g_1^\top g_2\geq g_2^\top g_2$ or $g_2^\top g_2\geq g_1^\top g_2\geq g_1^\top g_1$. If $g_i^\top g_i\geq g_i^\top g_j\geq g_j^\top g_j$, we set $d_k=-g_{j}$.
Then
\begin{eqnarray*}
H(x_{k+1})-H(x_k)&\leq&\max\{h_i(x_{k+1})-h_i(x_k),~h_j(x_{k+1})-h_j(x_k)\}\\
&\leq&\max\{\frac{1}{2L^{2}}g_j^\top A_ig_j-\frac{1}{L}g_i^\top g_j,~\frac{1}{2L^{2}}g_j^\top A_ig_j-\frac{1}{L}g_j^\top g_j\}\\
&\leq&\max\{\frac{1}{2L^{2}}g_j^\top A_ig_j-\frac{1}{L}g_j^\top g_j,~\frac{1}{2L^{2}}g_j^\top A_ig_j-\frac{1}{L}g_j^\top g_j\}\\
&\leq&\max\{\frac{1}{2L}g_j^\top g_j-\frac{1}{L}g_j^\top g_j,~\frac{1}{2L}g_j^\top g_j-\frac{1}{L}g_j^\top g_j\}\\
&=&-\frac{1}{2L}g_j^\top g_j=-\frac{L}{2}\norm{x_k-x_{k+1}}^{2}.
\end{eqnarray*}

Symmetrically, if $g_i^\top g_j>0$ and $g_j^\top g_j\geq g_i^\top g_j\geq g_i^\top g_i$, setting  $d_k=-g_{i}$ yields the same sufficient descent property.
\end{itemize}
\item  When $h_1(x_k)\neq h_2(x_k)$ and the quadratic equation (\ref{qe}) for $\gamma$ has no positive solution or has a positive solution $\gamma\geq1/L$, we have $h_i(x_{k+1})>h_j(x_{k+1})$  for $x_{k+1}=x_k+\beta_kd_k$, where $d_k=-\nabla h_i(x_k)$  and $\beta_{k}=\frac{1}{L}$. Moreover,
\begin{eqnarray*}
H(x_{k+1})-H(x_k)&=& h_i(x_{k+1})-h_i(x_k)\\
&=&\frac{1}{2L^{2}}g_i^\top A_ig_i-\frac{1}{L}g_i^\top g_i\\
&\leq&-\frac{1}{2L}g_i^\top g_i.
\end{eqnarray*}
Hence $H(x_k)-H(x_{k+1})\geq\frac{1}{2L}g_i^\top g_i\geq\frac{L}{2}\norm{x_k-x_{k+1}}^2$.
\item When $h_1(x_k)\neq h_2(x_k)$ and the quadratic equation (\ref{qe})  has a positive solution $\gamma<1/L$. With $\beta_{k}=\gamma$  and $d_k=-\nabla h_i(x_k)$,
 it is easy to see that the step size $\gamma$ makes $h_1(x_{k+1})=h_2(x_{k+1}).$ Then we have
\begin{eqnarray*}
H(x_{k+1})-H(x_k)&=&h_i(x_{k+1})-h_i(x_k)\\&=&\frac{1}{2}\gamma^2d_{k}^\top A_id_{k}+\gamma g_{i}^\top d_{k}\\
&\leq& \frac{1}{2}L\gamma^2g_i^\top g_i-\gamma g_i^\top g_i\\
&=& (\frac{L}{2}-\frac{1}{\gamma})\norm{x_k-x_{k+1}}^2, \end{eqnarray*}
 which further implies $H(x_k)-H(x_{k+1})\geq\frac{L}{2}\norm{x_k-x_{k+1}}^2$ due to $\gamma\leq\frac{1}{L}$.
\end{enumerate}
\end{proof}
\begin{rem}
It is worth to note that Step 3 in our algorithm is somehow similar to the retraction step in manifold optimization \cite{absil2009optimization}. In manifold optimization, in every iteration, each point is retracted to the manifold. In Step 3, every  point  is drawn to the curve that $h_1(x)=h_2(x)$.
\end{rem}

We will next show that in general a global sublinear convergence rate, in the same order with the gradient descent algorithm,  can also be theoretically guaranteed for Algorithm 1.

\begin{thm}
Assume that $x^*$ is an optimal solution. Then we have $$H(x_N)-H(x^*)\leq\frac{L}{N} \norm{x_{0}-x^*}^2.$$
That is, the required iterate number for $H(x_N)-H(x^*)\leq \epsilon$ is at most $O(1/\epsilon)$. \end{thm}
\begin{proof}
From the proof in Theorem \ref{lincon}, for any step size $\gamma\leq1/L$, we have
$$H(x_{k+1})-H(x_k)\leq-\gamma g^\top g+\frac{1}{2}L\gamma^2g^\top g\leq-\frac{\gamma }{2}g^\top g.$$
From the convexity of $H(x)$ and $g\in\partial H(x_k)$, we have
\begin{eqnarray*}
H(x_{k+1})&\leq& H(x_k)-\frac{\gamma }{2}g^\top g\\
&\leq&H(x^*)+g^\top(x_k-x^{*})-\frac{\gamma }{2}g\top g\\
&=&H(x^*)+\frac{1}{2\gamma}\left(\norm{x_k-x^*}^2-\norm{x_k-x^*-\gamma g}^2\right)\\
&=&H(x^*)+\frac{1}{2\gamma}\left(\norm{x_k-x^*}^2-\norm{x_{k+1}-x^*}^2\right).
\end{eqnarray*}
Since $H(x_{k+1})\geq H(x^*)$, we have $\norm{x_k-x^*}^2-\norm{x_{k+1}-x^*}^2\geq0$. Let us use  indices $i_k,~k=0,\ldots,K$ to denote the indices in Steps 1 and 2. By noting that $\gamma=1/L$, we have
$$H(x_{i_{k+1}})\leq H(x^*)+\frac{L}{2}\left(\norm{x_{i_k}-x^*}^2-\norm{x_{i_k+1}-x^*}^2\right).$$ Note that every Step 3 is followed by S tep 1. Hence $N\leq 2K+1$.  Adding the above inequalities from $i_0$ to $i_K$, we have
\begin{eqnarray*}
&&\sum_{k=0}^KH(x_{i_{k}})-H(x^*)\\
&\leq&\frac{L}{2}\sum_{k=0}^K\left(\norm{x_{i_k}-x^*}^2-\norm{x_{i_k+1}-x^*}^2\right)\\ &\leq&\frac{L}{2}\left(\norm{x_{i_0}-x^*}^2-\norm{x_{i_K+1}-x^*}^2+\sum_{k=1}^K\left(-\norm{x_{i_{k-1}+1}-x^*}^2+\norm{x_{i_k}-x^*}^2\right)\right)\\ &\leq&\frac{L}{2}(\norm{x_{i_0}-x^*}^2-\norm{x_{i_K+1}-x^*}^2)\\
&\leq&\frac{L}{2}\norm{x_{i_0}-x^*}^2\\
&\leq&\frac{L}{2}\norm{x_{0}-x^*}^2,
\end{eqnarray*}
where in the second inequality we use the fact,
$$-\norm{x_{i_{k-1}+1}-x^*}^2+\norm{x_{i_k}-x^*}^2\le-\norm{x_{i_{k-1}+1}-x^*}^2+\norm{x_{i_k-1}-x^*}^2\le\cdots\le0 .$$ Since $H(x_k)$ is non-increasing, by noting that $N\leq 2K+1$, we have
\begin{eqnarray*}
H(x_N)-H(x^*)&\leq&\frac{1}{K+1}\sum_{k=0}^KH(x_{i_k})-H(x^*)\\
&\leq& \frac{L}{N} \norm{x_{0}-x^*}^2.
\end{eqnarray*}
\end{proof}
\begin{algorithm}[!ht]
\caption{Line search with the modified Armijo rule for Problem (M)}
\label{alg2}
\begin{algorithmic}[1]
\Require Parameters in the minimax problem (M) and $\rho>0$ \State Initialize  $x_0$\For {$k=0,1,\ldots, $}
\If {$h_1(x_k)>h_2(x_k)+\rho$}
\State set $d_{k}=-\nabla h_{1}(x_k)$
\ElsIf {$h_1(x_k)<h_2(x_k)-\rho$}
\State set $d_{k}=-\nabla h_{2}(x_k)$
\Else
\State set $d_{k}$ corresponding to Theorem \ref{descentthm}, item 2
\EndIf
\If {termination criterion is met}
\State \Return
\EndIf
\State Choose a step size $\beta_k>0$ according to the modified Armijo rule (\ref{armijorule})
\State Update $x_{k+1}=x_k+\beta_kd_{k}$
\EndFor
\end{algorithmic}
\end{algorithm}

\subsection{Line search with the modified Armijo rule}
An alternative way to choose the step size in the classical gradient descent type methods  is the line search with the Armijo rule. A natural thought is then to extend the Armijo rule in our minimax problem (M) as  in the proposed Algorithm \ref{alg2}.
In particular, we set the following modified Armijo rule to choose the smallest nonnegative integer $k$ such that the following inequality holds for the step size $\beta_k=\xi s^k$ with $0<\xi\leq1$ and $0<s<1$,
\begin{equation}\label{armijorule}
 f(x_k+\beta_k p_k)\leq f(x_k)+\sigma\beta_k p_k^\top g,
\end{equation}
where $0\leq\sigma\leq0.5,$ $g=-d$ and  $d$ is the steepest descent direction defined in Theorem \ref{descentthm}. Particularly, we set the search direction $p_k=d$ at iterate $k$. Our numerical result in the next section shows that Algorithm \ref{alg2} has a comparable performance when compared with  (or even better than) Algorithm \ref{alg1}.
For the sake of completeness, we present the convergence result  for Algorithm \ref{alg2} in the following.
 Before that, we generalize the definition of a critical point to a $(\rho,\delta)$ critical point.
\begin{dfn}\label{defsubg}
A point $x$ is called a $(\rho,\delta)$ critical point of $H(x)=\max \{h_1(x),h_2(x)\}$ if  $\exists\norm{g}<\delta$, for some $g\in\partial H_{\rho}(x)$, where $\partial H_{\rho}(x)$ is defined as follows:\begin{enumerate}
\item $\partial H_{\rho}(x)=\{\alpha \nabla h_1(x)+(1-\alpha)\nabla h_2(x):\alpha\in[0,1]\}$, if $|h_1(x)-h_2(x)|\leq \rho$;
\item $\partial H_{\rho}(x)=\{\nabla h_1(x)\}$, if $h_1(x)-h_2(x)> \rho$;
\item  $\partial H_{\rho}(x)=\{\nabla h_2(x)\}$, if $h_2(x)-h_1(x)> \rho$.  \end{enumerate}
\end{dfn}
The following proposition shows the relationship of a critical point  and  a $(\rho,\delta)$ critical point. As this result is pretty obvious, we omit its proof.
\begin{prop}\label{epdel}
Assume that  $\{x_k\}$ is a sequence in $\mathbb{R}^n$ and that $(\rho^t,\delta^t)\rightarrow(0,0)$, for $t\rightarrow\infty$  and that there exists a positive integer  $K{(t)}$, such that  $x_k$ is a $(\rho^t,\delta^t)$ critical point of $H(x)$ for all $k\geq K(t)$ and $t\ge1$. Then, every accumulation point of the  sequence  $\{x_k\}$ is  a critical point of $H(x)$.
\end{prop}

Slightly different from Algorithm \ref{alg1}, our goal in Algorithm \ref{alg2} is to find a $(\rho,\delta)$ critical point. With Proposition \ref{epdel}, we conclude that Algorithm \ref{alg2} outputs a solution that is sufficiently close to a critical point of $H(x)$.

\begin{thm}\label{conver2}
Assume that  i) $d={\rm argmin}_{y\in\partial  H_{\rho}(x_k)} \norm{y}$ with $\rho>0$,ii) the termination criterion is $\norm{d}<\delta$ for some $\delta>0$ and iii) $x^*$ is an optimal solution. Then for any given  positive numbers $\rho$ and $\delta$,   Algorithm \ref{alg2} generates a $(\rho,\delta)$ critical point in at most
$$\frac{H(x_0)-H(x^*)}{\sigma s\min\{1/L,\xi,\frac{\rho}{2G^2}\}\delta^2} $$
iterations, where $G$ is some positive constant only depending on the initial point and problem setting.
\end{thm}
\begin{proof}
 Consider the following different cases with $\norm{d}\geq\delta$.
\begin{enumerate}
\item If $|h_1(x_k)-h_2(x_k)|<\rho$, then as assumed  $\norm{d}>\delta$ and  from Theorem \ref{descentthm}, we know that $d={\rm argmin}_{\alpha\in[0,1]}  \norm{\alpha\nabla h_1(x_k)+(1-\alpha)\nabla h_2(x_k)}$ is just the parameter $\alpha$ which we choose in Algorithm \ref{alg2}.
It suffices to show that the step size $\beta_k$ is bounded from below such that
$$H(x_{k+1})-H(x_k)\leq -\sigma\beta_kd^\top d.$$ This further suffices
to show that  $\beta_k$ is bounded from below such that for $i=1$ or $2$,
\begin{equation}\label{hineq}
h_i(x_{k+1})-h_i(x_k)= -\beta_k\nabla h_i(x_k)^\top d+\frac{1}{2}\beta_k^{2}d^\top A_id\leq -\sigma\beta_kd^\top d.
\end{equation}
By noting that $\nabla h_i^\top d\geq d^\top d$ from Remark \ref{smallestnorm}, the  second inequality in (\ref{hineq}) holds true for all $\beta_k\le2(1-\sigma)/L $.  Then the step size chosen by the modified Armijo rule satisfies $\beta_k\geq s\min\{2(1-\sigma)/L,\xi\}$, which further implies that
$$H(x_k)-H(x_{k+1})\geq \sigma\beta_kg^\top g=\sigma\beta_k \norm{g}^2\geq \sigma s\min\{2(1-\sigma)/L,\xi\}\delta^2.$$
\item If $h_1(x_k)-h_2(x_k)>\rho \ $and $\norm{\nabla h_1(x_{k})}>\delta$, we have $g=\nabla h_1(x_{k})$.
Because $H(x_k)$ is decreasing, under Condition \ref{con1}, $h_1(x_k)+ h_2(x_k)=\frac{1}{2}x_k^\top (A_1+A_2)x_k+(a_{1}+ a_2)^\top x_k\leq 2h_{1}(x_k)=2H(x_k)\leq2H(x_0) $ and thus $x_{k}$ is bounded due to $A_1+A_2=2(Q_1+\frac{\lambda_1+\lambda_2}{2}Q_2)\succ0$. This further implies that $\nabla h_i(x_k)=A_ix_{k}+b_i$  is bounded for all $k$. So there exists some positive constant only depending on the initial point and problem parameters such that $\norm{\nabla h_i(x_k)}\leq G$, $i=1,$.  Hence $\norm{d}\leq G$ because $d$ is a convex combination of $\nabla h_1(x_k)$ and $\nabla h_2(x_k)$.
Then we have
$$h_1(x_{k+1})-h_1(x_{k})\leq -\beta_k\nabla h_1(x_{k})^\top d+\frac{1}{2}\beta_k^{2}d^\top A_1d$$
and for any $\beta_k\leq1/L,$\begin{eqnarray*}
h_2(x_{k+1})-h_2(x_{k})&\leq& -\beta_k\nabla h_2(x_{k})^\top g+\frac{1}{2}\beta_k^{2}g^\top A_ig\\
&\leq& \beta_kG\norm{g}+\frac{1}{2}\beta_k^2L\norm{g}^2\\
&\leq&\beta_kG^2(1+\frac{1}{2}\beta_kL)\\
&\leq&\frac{3}{2}\beta_kG^2.
\end{eqnarray*}
On the other hand,  when $\beta_{k}\leq1/L$,
$$h_1(x_{k+1})-h_1(x_{k})\leq -\beta_k\nabla h_1(x_{k})^\top g+\frac{1}{2}\beta_k^{2}g^\top A_1g\leq -\beta_kg^\top g+\frac{1}{2}\beta_k^{2}Lg^\top g\leq-\frac{1}{2}\beta_kg^\top g.$$
Note that for all $\beta_k\leq\frac{\rho}{2G^2}$,    $\frac{3}{2}\beta_kG^2+\frac{1}{2}\beta_kg^\top  g\leq\rho$.
Thus for   $\beta_k\leq\min\{1/L,~\frac{\rho}{2G^2}\}$,
we have
\begin{eqnarray*}
&&h_1(x_{k+1})\leq h_1(x_k)-\frac{1}{2}\beta_kg^\top g,\\
&&h_2(x_{k+1})\leq h_2(x_k)+\frac{3}{2}\beta_kG^2\leq h_1(x_k)-\rho+\frac{3}{2}\beta_kG\leq h_1(x_k)-\frac{1}{2}\beta_kg^\top g.
\end{eqnarray*}
Hence we have
\begin{eqnarray*}
H(x_{k+1})-H(x_{k})&=&\max\{h_1(x_{k+1}),h_2(x_{k+1})\}-h_1(x_k)\\
&=&\max  \{h_1(x_{k+1})-h_1(x_{k}),~h_2(x_{k+1})-h_1(x_{k})\}\\
&\leq&\max  \{h_1(x_{k+1})-h_1(x_{k}),~h_2(x_{k+1})-h_2(x_{k})\}\\
&\leq&-\frac{1}{2}\beta_kg^\top g.
\end{eqnarray*}
So the Armujo rule implies  $\beta_k\geq s\min\{1/L,\xi,\frac{\rho}{2G^2}\}$, i.e., $\beta_k$ is lower bounded. Then according to the modified Armijo rule, we have
\begin{equation}\label{itebd2}
H(x_{k})-H(x_{k+1})\geq\sigma\beta_kg^\top g\geq\sigma s\min\{1/L,\xi,\frac{\rho}{2G^2}\}\delta^2.
\end{equation}
\item Symmetrically, the case with $h_2(x_{k})-h_1(x_{k})>\rho$  yields the same result as in (\ref{itebd2}).
\end{enumerate}
The above three cases show that $H(x_{k})-H(x_{k+1})\geq\sigma s\min\{1/L,\xi,~\frac{\rho}{2G^2}\}\delta^2$ (as $1-\sigma\geq1/2$, the decrease in case 1 also admits this bound). Since the decrease in each iterate is larger than $\sigma s\min\{1/L,\xi,\frac{\rho}{2G^2}\}\delta^2$, the total iterate number is bounded by
$$\frac{H(x_0)-H(x^*)}{\sigma s\min\{1/L,\xi,\frac{\rho}{2G^2}\}\delta^2}.$$\end{proof}

At the current stage, we cannot demonstrate a theoretical convergence rate for Algorithm \ref{alg2}  as good as the sublinear rate $O(1/\rho)$ for Algorithm 1 in Theorem \ref{lincon}. But our numerical tests show that Algorithm \ref{alg2}  converges as fast as Algorithm \ref{alg1}. Proposition \ref{epdel} and Theorem \ref{conver2} offer our main convergence result for Algorithm \ref{alg2} as follows.
\begin{thm}
Assume that $(\phi_k,\psi_{k})\rightarrow0$ and  that $\{x^{(k)}\}$ is a sequence of  solutions generated by Algorithm \ref{alg2} with $\rho=\phi_k$ and $\delta=\psi_{k}$. Then
any accumulation point of $\{x^{(k)}\}$ is an optimal solution of problem $\rm(M)$.
\end{thm}

\section{Numerical tests}
In this section, we illustrate the efficiency of our algorithm with numerical experiments. All the numerical tests were implemented in
Matlab 2016a, 64bit and were run on a Linux machine with 48GB
RAM, 2600MHz cpu and 64-bit CentOS release 7.1.1503.  We compare both Algorithms \ref{alg1} and \ref{alg2}  with the ERW algorithm in \cite{pong2014generalized}.
We  disable the parallel setting in the Matlab for fair comparison. If the parallel setting is allowed, our algorithm has a significant improvement, while the ERW algorithm does not.

We use the following same  test problem as \cite{pong2014generalized} to show the efficiency of our algorithms,\begin{eqnarray*}
\rm(IP)&\min&x^\top Ax-2a^\top x\\
&\rm s.t&c_1\leq x^\top Bx\leq c_2,
\end{eqnarray*}
where $A$ is an $n\times n$  positive definite matrix and $B$ is an  $n\times n$ (nonsingular) symmetric indefinite matrix.
We first reformulate problem (IP) to  a formulation of problem (P) in the following procedure, which is motivated from \cite{pong2014generalized} (the proof in \cite{pong2014generalized} is also based on the monotonicity of $\gamma(\lambda)$, which is defined in Section 2.1), in order to  apply the CQR for problem (P) and  then invoke Algorithms \ref{alg1} and \ref{alg2} to solve the CQR.
\begin{thm}\label{procedure}
  Let $x_0=-A^{-1}a$. Then
the followings hold.  \begin{enumerate}
\item If $x_0^\top Bx_0<c_1$, problem $\rm(IP)$ is equivalent to
$${\rm(IP_1)}~~~\min\{ x^\top Ax-2a^\top x:{\rm s.t.}~c_1\leq x^\top Bx\};$$
\item Else if $c_1\leq x_0^\top Bx_0\leq c_2$, problem $\rm(IP)$  admits an interior solution $x_0$;
\item Otherwise $c_2<x_0^\top Bx_0$,  problem $\rm(IP)$ is equivalent to$${\rm(IP_2)}~~~\min\{ x^\top Ax-2a^\top x:{\rm s.t.}~ x^\top Bx\leq c_2\}.$$\end{enumerate}
\end{thm}
\begin{proof}
Item 2 is obvious. Item 1 and Item 3 are symmetric. So in the following, we only prove Item 1.

In our problem set, matrix $A$ is positive definite and $B$ is indefinite. Hence, in the definition $I_{PSD}=\{\lambda:Q_1+\lambda Q_2\succeq0\}$, we have $\lambda_1<0$, $\lambda_2>0$. Thus from Case 1 in Section 2.2.2 in \cite{pong2014generalized} we know, when $x_0^\top Bx_0<c_1$,  problem $\rm(IP)$ is equivalent to
$$ {(\rm EP_1)}~~~\min\{ x^\top Ax-2a^\top x:{\rm s.t.}~c_1=x^\top Bx\}.$$
Since  $x_0^\top Bx_0<c_1$, the optimal solution of $(\rm IP_1)$ must be at its boundary \cite{more1993generalizations}. This further yields  that problem $\rm(IP)$ is equivalent to $(\rm IP_1)$.
\end{proof}

Theorem 4.1 helps us  solve problem (IP) as an inequality constrained GTRS instead of solving two GTRS with equality constraints.
Before showing the numerical results, let us illustrate some functions used in our initialization. To obtain the CQR, the generalized eigenvalue problem is solved  by \verb"eigifp"  in Matlab, which was developed in \cite{golub2002inverse} for computing the maximum generalized eigenvalues for sparse definite matrix pencils.
In our numerical setting \verb"eigifp" is usually faster than the Matlab function \verb"eigs", though \verb"eigs" will outperform  \verb"eigifp" when the condition number is large or the density is low. We use the Matlab command \verb"sprandsym(n,density,cond,2)" and \verb"sprandsym(n,density)" to generate $Q_1$ and $Q_2$. We set the density of matrices  at 0.01 and use three levels of condition number for matrix $Q_1$, i.e., 10, 100 and 1000 and, in such settings,   \verb"eigifp" always dominates  \verb"eigs" (this may be because {\verb"eigs"} is developed for computing extreme generalized eigenvalues for arbitrary matrices and does not utilize the definiteness and symmetry properties of the matrix pencils in our problem setting).
In general, the main cost in estimating $L$ is to compute the maximum eigenvalues of matrices $A_1$ and $A_2$, which may be time consuming for large-scale matrices. To conquer this difficulty, we can  estimate a good upper bound with very cheap cost instead. Specially, we can run the function \verb"eigifp" with precision 0.1, which is much more efficient than  computing the true maximum eigenvalue with  \verb"eigifp", and, assuming $M$ is the output,  $M+0.1$ is then a good upper bound for $L$. In our numerical tests, we just use  \verb"eigifp" to estimate $L$ since our main goal is to illustrate the efficiency of Algorithm 2. In Algorithm \ref{alg1}, to avoid some numerical accuracy problem, we approximate  $h_1(x_k)=h_2(x_k)$ by $|h_1(x_k)-h_2(x_k)|/(|h_1(x_k)|+|h_2(x_k)|)\leq\epsilon_1$. Also we use $|h_1(x_k)-h_2(x_k)|/(|h_1(x_k)|+|h_2(x_k)|)\leq\epsilon_1$ instead of  $|h_1(x_k)- h_1(x_k)|\leq\rho$ in Algorithm 2 for stableness consideration. In our numerical tests for both Algorithms \ref{alg1} and \ref{alg2}, we use the following termination criteria  (if any one of the following three conditions is met, we terminate our algorithm), which are slightly different from the presented algorithms for robust consideration: \begin{enumerate}
\item $H(x_{k-1})-H(x_k)<\epsilon_2$,
\item $|h_1(x_k)-h_2(x_k)|/(|h_1(x_k)|+|h_2(x_k)|)\leq\epsilon_1$, $\norm{\alpha\nabla h_1(x_k)+(1-\alpha)\nabla h_2(x_k)}\leq \epsilon_3$,
\item $\norm{\nabla h_i(x_k)}\leq\epsilon_3$ and $|h_1(x_k)-h_2(x_k)|/(|h_1(x_k)|+|h_2(x_k)|)>\epsilon_1$, where $i\neq j$ and $i,j\in\{1,2\},ㄛ$
\end{enumerate}
where $\epsilon_1,~\epsilon_2$ and $\epsilon_3>0$ are some small positive numbers for termination of the algorithm.   Particularly, we set   $\epsilon_1=10^{-8},~\epsilon_2=10^{-11}$ and $\epsilon_3=10^{-8}$ in Algorithm 1, and $\epsilon_1=10^{-8},~\epsilon_2=10^{-11},~\epsilon_3=10^{-8}$,  $\sigma=10^{-4}$ and $\xi=1$ (for the modified Armijo rule) in Algorithm 2.

To improve the accuracy of the solution, we apply the Newton refinement process in Section 4.1.2 in \cite{adachi2016eigenvaluebased}. More specifically, assuming $x^*$ is the solution returned by our algorithm, we update $x^*$ by
$$\delta=\frac{(x^*)^\top Bx^*}{2\norm{Bx^*}^2}Bx^*, ~~x^*=x^*-\delta.$$In general, the ERW algorithm can achieve a higher precision than our method (after the Newton refinement process); the precision in their method is about $10^{-14}$, while ours is slightly less precise than theirs. Letting $v_1$ denote the optimal value of ERW algorithm and $v_2$ denote the optimal value of our algorithm, we  have at least ${|v_2-v_1|}/{|v_1|}\approx10^{-10}$ for most cases.
The iteration number reduces to 1/5 if we reduce the precision of from $\epsilon_1=10^{-8},~\epsilon_2=10^{-11},~\epsilon_3=10^{-8}$ to $\epsilon_1=10^{-5},~\epsilon_2=10^{-8},~\epsilon_3=10^{-5}$. This observation seems reasonable as our method is just a first order method.
 \begin{table}[h!]
\scriptsize
\tabcolsep=1pt
 \vspace*{0.0in} \centering
\caption{\label{tab2} Numerical results for positive definite $A$ and indefinite $B$}
\begin{tabular}{lc|cc|cc|c|cc|c||cc|cc|c|cc|c}
\hline
\multirow{3}*{cond}&\multirow{3}*{n} &\multicolumn{8}{c}{Easy Case}&\multicolumn{8}{c}{Hard Case 1}\\
\cline{3-17}
&&\multicolumn{2}{c|}{Alg1}&\multicolumn{2}{c|}{Alg2}&\multirow{2}*{$\rm time_{\text eig}$}&\multicolumn{3}{c||}{ERW}&\multicolumn{2}{|c|}{Alg1}&\multicolumn{2}{c|}{Alg2}&$\multirow{2}*{$\rm time_{\text eig}$}$&\multicolumn{3}{|c}{ERW}\\

&&iter&time&iter&time&&iter&time&fail&iter&time&iter&time&&iter&time&fail\\
\hline
 10& $10000$&90&\textbf{1.03} & 109.3& 1.24&1.45&5.9 &4.89&0&1490&16.7&609.6&\textbf{6.81}&1.19&6&11.1&1\\
 10& $20000$&52&\textbf{2.82} & 72.2&3.91&9.20&6.8&25.1&0& 530.3&27.9&313.9&\textbf{16.7}&7.56&6.5&53.9&0  \\
 10& 30000&60.9&\textbf{9.81}  &83.2&13.4&25.2&6.6 &75.0&0&1014.6&157&270.6&\textbf{41.0}&30.1&7.3&170&1 \\
 10& $40000$&58.3& \textbf{17.1}  &95.2&27.8&49.7&6.8&153&0& 1866.4&520&782.7&\textbf{219}&54.0&7.1&356&1\\
 \hline
 100& $10000$&417.7&\textbf{4.26}& 424.9& 4.34&3.99&5.9 &11.4&0&3328.2&33.9&1131.6&\textbf{13.6}&3.63&5.7&24.6&3\\
 100& $20000$&474.3&24.6 &342.4&\textbf{17.8}&18.4&6.1&69.4&0&6494.9&350&1410&\textbf{76.8}&42.2&6.4&123&5   \\
 100& $30000$&196.9&28.0  &162.1&\textbf{23.1}&51.8&6.2&147&0&2836.6&420&1197.9&\textbf{176}&44.2&5.2&388&0  \\
 100& $40000$&135.8&40.1    &114.7&\textbf{33.9}&153.6&6.3&309&0&906.7&257&506.1&\textbf{143}&173.5&6.5&639&0 \\
 \hline
 1000& $10000$&4245&44.7& 1706.7& \textbf{17.8}&14.2&5.3&56.7&0&25982.6&261&5090.7&\textbf{51.3}&24.0&5.75&81.1&6\\
 1000& $20000$&4177.3&216&1182.7&\textbf{61.2}&70.8&6.10&368&0&26214.8&1360&2726.8&\textbf{139}&98.1&5.8&346&5   \\
 1000& $30000$&2023.8&289  &813.7&\textbf{116}&189&5.9 &1220&0&15311.4&2190&2591.9&\textbf{385}&195&5.8&1530&3 \\
1000 &$40000$&2519.8&652  &1003&\textbf{301}&640.9&6.8 &2960&0&8735.8&3060&1343&\textbf{1020}&853&6.25&3280&2 \\
\hline
\end{tabular}
\end{table}

We report our numerical results in Table 1.  We use ``Alg1'' and ``Alg2'' to denote Algorithms \ref{alg1} and \ref{alg2}, respectively.   For each $n$ and each condition number, we generate 10 Easy Case and 10 Hard Case 1 examples.
Please refer to Table 1 in \cite{pong2014generalized} for the detailed definitions of Easy Case and Hard Cases 1 and 2. There is a little difference about the definitions of easy and hard cases between \cite{pong2014generalized} and  \cite{more1993generalizations}. Our analysis in the above sections uses the definitions in  \cite{more1993generalizations}. In fact, the Easy Case and Hard Case 1 are the easy case  and Hard Case 2 is the hard case mentioned in the above sections and \cite{more1993generalizations}. We use the notation ``time" to denote the average CPU time (in unit of second) and ``iter'' to denote the average iteration numbers for all the three algorithms. For ``Alg1'' and ``Alg2'', ``time" is  just the time for Algorithms \ref{alg1} and \ref{alg2}, respectively. The notation ``$\rm time_{eig}$" denotes the average CPU time for computing the generalized eigenvalue for our algorithm. So the total time for solving problem (P) should be the summation of the time of reformulate (P) into (M) and the time of Algorithm 1 or 2, whose main cost is just $\rm ``time"+\rm ``time_{eig}"$. And ``fail"  denotes the failure times in the 10 examples in each case for the ERW algorithm. One reason of the failures may be that the ERW algorithm terminates in 10 iterations even when it does not find a good approximated solution. We point out that for randomly generated test examples, our method always succeeds in finding an approximated solution to prescribed precision while  the ERW algorithm fails frequently in Hard Case 1.   Another  disadvantage of the ERW algorithm is the requirement of an efficient prior estimation of the initialization, which is unknown in general. In our numerical test, we assume that such an initialization is given as the same as \cite{pong2014generalized} does.

We also need to point out that in the Hard Case 2,  our algorithms do not outperform the ERW algorithm which uses the  shift and deflation technique.
 The main time cost of shift and deflate operation
is the computation of the extreme generalized eigenvalue of the matrix pencil $(A,B)$ and its corresponding generalized eigenvectors.
In the test instances, as the dimension of the eigenspace of the extreme generalized eigenvalue is  one, the   shift and deflation technique directly finds the optimal solution by calling  \verb"eigifp" once. Our algorithm reduces to an unconstrained quadratic minimization in Hard Case 2. However, the condition number of this unconstrained  quadratic minimization is so large that our algorithm performs badly as the classical gradient method. To remedy this disadvantage, we can add a step with almost free-time cost that claims that   either  we are in Hard Case 2 and output an optimal solution or we are in Easy Case or Hard Case 1. Recall that the hard case (or equivalently, Hard Case 2)  states that $b_1+\lambda^*b_2$ is orthogonal to the null space of $Q_1+\lambda^*Q_2$ which means that  $\lambda^*$  must be a boundary point of $I_{PSD}$. Suppose $\lambda_i=\lambda^*$. Then we must have  that  $x^{*}=\arg\min H(x)$ and $H(x^*)=h_i(x^{*})$ for some $i$=1 or 2.   In fact, if $\nabla h_i(x)=0$ and $h_i(x)\ge h_j(x),~j\in\{1,2\}/\{i\}$ for some $x$, then $x$ is optimal and we are in the hard case. So $\nabla h_i(x)=0$ and $h_i(x)\ge h_j(x)$ is sufficient and necessary for $x$ to be optimal to (M)  and be in the hard case. Hence we can construct an optimal solution for problem (M) as  $\bar x=(Q_1+\lambda_{i} Q_2)^\dagger(b_1+\lambda _{i}b_2)+\sum_i^k \alpha_j v_j$  (where $A^\dagger$ denotes the \textit{Moore--Penrose} pseudoinverse of $A$) if $v_{j},j=1,\ldots,k$  are the generalized eigenvectors of matrix pencil ($Q_1,Q_2$) with respect to the generalized eigenvalue $\lambda_i$ such that $h_i(\bar x)\ge h_j(\bar x)$ and $\alpha\ge0$. This equals to identifying if a  small dimensional convex quadratic programming problem (with variable $\alpha$) has an optimal value less than $h_i((Q_1+\lambda_{i} Q_2)^\dagger(b_1+\lambda _{i}b_2))$. And if such $\alpha$ does not exist, we are in the easy case (or equivalently, Easy Case or Hard Case 1). This technique is very similar to the shift and deflation technique in \cite{fortin2004trust,pong2014generalized}. Hence we can  solve Hard Case 2 within almost the same CPU time as the ERW algorithm. So we do not make further comparison for Hard Case 2.

Our numerical tests show that  both Algorithms 1 and 2 are much more efficient than the ERW algorithm in
Easy Case and for most cases in Hard Case 1.  The efficiency of our algorithms is mainly due to that  we only call the generalized eigenvalue solver once and every iteration only involves several matrix vector products (which are very cheap for sparse matrices).  We also note that, in Easy Case, Algorithm 1 is faster than Algorithm 2 when the condition number is small and slower than Algorithm 2 when the condition number is large. This may be because that  Algorithm 2  is equipped with the modified Armijo rule, which makes it more aggressive in choosing the step size and thus yields a fast convergence. In Hard Case 1,
Algorithm 2 is still much more efficient than the ERW algorithm while Algorithm 1 is slower than the ERW algorithm in about half the cases. This is because Algorithm 2 has a moderate iterate number due to the aggressiveness in choosing the step size and Algorithm 1 has a much large iterate number for these cases.  Moreover, our algorithms always succeed, while the ERW algorithm fails frequently in Hard Case 1.
A more detailed analysis with condition number for Algorithm 1 will be given in the following.

We note that several examples (of the 10 examples) in Easy Cases admit a much larger iteration number than average. This motivates us to analyze the main factor that affects the convergence rate (reflected by the iteration number) of Algorithm 1 (the analysis for Algorithm 2 seems hard  due to the non-smoothness of the problem). We then find that the main factor  is $\sqrt{\lambda_{\max \alpha}/2\lambda_{\min nnz\alpha}^2}$, as evidenced by the fact that examples in Easy Case and Hard Case 1 with more iterates all have a larger $\sqrt{\lambda_{\max \alpha}/2\lambda_{\min nnz\alpha}^2}$, where  $\lambda_{\max \alpha}$ denotes the maximum eigenvalue of matrix $\alpha A_1+(1-\alpha)A_2$ and $\lambda_{\min nnz\alpha}$ denotes the smallest nonzero eigenvalue of matrix $\alpha A_1+(1-\alpha)A_2$  with $\alpha$ being defined in Theorem \ref{descentthm} in the last iteration. In fact, when $x^k\rightarrow x^*\in\{x:\partial H(x)=0\}$ (in our examples, the optimal solution is unique), let the value of $\alpha$ at iterate $k$ be $\alpha^k$, then
$\alpha^k\rightarrow \alpha^*$, where $\alpha^*$  is the solution of $\alpha\nabla h_1(x^*)+(1-\alpha)\nabla h_2(x^*)=0$.
From the definition of KL exponent,  we have
$$C\times \min_\alpha\norm{\alpha\nabla h_1(x^k)+(1-\alpha)\nabla h_2(x^k)}\ge|H(x^k)-H(x^*)|^{1/2}.$$
Intuitively, the smallest value of  $C$ should be at least $$\frac{|H(x^k)-H(x^*)|^{\frac{1}{2}}}{ \min_\alpha\norm{\alpha\nabla h_1(x^k)+(1-\alpha)\nabla h_2(x^k)}}\rightarrow\frac{|\alpha (h_{1}(x^k)-h_2(x^*))+(1-\alpha) (h_{1}(x^k)-h_2(x^*))|^{\frac{1}{2}}}{ \min_\alpha\norm{\alpha\nabla h_1(x^k)+(1-\alpha)\nabla h_2(x^k)}}$$ which is upper bounded by $\sqrt{\lambda_{\max \alpha}/2\lambda_{\min nnz\alpha}^2}$. Thus, the asymptotic value of $C$ can be  roughly seen as $\sqrt{\lambda_{\max \alpha}/2\lambda_{\min nnz\alpha}^2}$. Hence both Easy Case and Hard Case 1 admit local linear convergence and the convergence rate is
$$\left(\sqrt{1-\frac{1}{2C^2L}}\right)^k= \left(\sqrt{1-\frac{\lambda_{\min nnz\alpha}^2}{L\lambda_{\max\alpha}}}\right)^k$$
from Theorem \ref{lincon}.
We also observe  from our numerical tests that in most cases the values of   $\lambda_{\max\alpha}$ are similar and that $\lambda_{\min nnz\alpha}$ in Easy Case is much larger than $\lambda_{\min nnz\alpha}$ in Hard Case 1 and $\lambda_{\max\alpha}$ in Easy Case is very close to $\lambda_{\max\alpha}$  in Hard Case 1. Hence,   $\sqrt{1-\lambda_{\min nnz\alpha}^2/(L\lambda_{\max\alpha})}$ in Easy Case  is usually much smaller than that in Hard Case 1.   (As $Q_2$ is random in our setting, the larger the condition number of $Q_1$ is, the larger expectation of  $\sqrt{1-\lambda_{\min nnz\alpha}^2/(L\lambda_{\max\alpha})}$ is.) This  explains why the condition number of matrix $Q_1$ measures, to a large degree, the hardness of our algorithms in solving problem (M).
  Since Easy Case has a smaller $\sqrt{1-(\lambda_{\min nnz\alpha}^2/L\lambda_{\max\alpha})}$ than Hard Case 1 for the same condition number and problem dimension, Easy Case can be solved faster than Hard Case 1.
   This coincides
with our numerical results, i.e., Easy Case admits a smaller iterate number than  Hard Cases 1.

We also tried to apply MOSEK \cite{mosek2017mosek} to solve the CQR. But our numerical results showed that MOSEK is much slower than both our algorithms and the ERW algorithm, which took about 833 seconds for Easy Case and 960 second for Hard Case 1 with   $ n=10000$ and $\rm cond=10$. So we do not run further numerical experiments with MOSEK. We also tested the SOCP reformulation \cite{ben2014hidden} under the simultaneous digonalization condition of the quadratic forms of the GTRS and the DB algorithm in \cite{salahi2016efficient} based on the simultaneous digonalization condition of the quadratic forms. The simultaneous digonalization condition naturally holds for problem (IP) when $A$ is positive definite. Our preliminary result shows that our method is much more efficient than the two methods based on  simultaneous digonalization when $n\geq10000$ and density$=0.01$ and thus we also do not report the numerical comparison in this paper. We believe this is mainly because the  simultaneously digonalization procedure of the matrices involves  matrix inverse, matrix matrix product, a full Cholesky decomposition and a spectral decomposition (of a dense matrix),  which is more time consuming than the operations of  matrix vector products in our algorithm.  Hence we do not report the numerical results based on the simultaneous digonalization technique.

\section{Concluding remarks}
In this paper, we have derived a simple convex quadratic reformulation for the GTRS, which only involves a linear objective function and two convex quadratic constraints under mild assumption.
 We further reformulate the   CQR to an unconstrained minimax problem under Condition \ref{con1}, which is the case of interest. The minimax reformulation is a well structured convex, albeit non-smooth, problem.  By investigating its inherent structure, we have proposed two efficient matrix-free algorithms to solve this minimax reformulation. Moreover, we have offered a theoretical guarantee of global sublinear convergence rate for both algorithms and  demonstrate a local linear convergence rate for Algorithm 1 by proving  the KL property for the minimax problem with an exponent of $1/2$ under some mild conditions.
Our numerical results have demonstrated clearly out-performance of our algorithms over the state-of-the-art algorithm for the GTRS.

As for our future research, we would  like to show whether the CQR and the minimax reformulation and the algorithms for the minimax problem can be extended to solve GTRS  with additional linear constraints.
As  the analysis in numerical section indicates that our algorithms have similar performance with unconstrained quadratic minimization, i.e., both algorithms admit a locally linear convergence rate with the steepest descent method, we would like to generalize existing algorithms that are efficient in solving unconstrained quadratic minimization to solve our minimax reformulation, e.g., the conjugate gradient method or Nesterov's accelerated gradient descent algorithm.
Another line of future research is to investigate whether our algorithm can be extended to general minimax problems with more (finite number of) functions. It is also interesting to verify whether the KL property still holds and whether the KL exponent is still $1/2$ when more functions are involved.

\section*{Acknowledgements}
This research was partially supported by Hong Kong Research Grants Council under
Grants 14213716 and 14202017. The second author is also grateful to the support from Patrick Huen
Wing Ming Chair Professorship of Systems Engineering and Engineering Management. The authors would also like to thank Zirui Zhou and Huikang Liu  for their insightful discussions.

\bibliographystyle{abbrv}
\bibliography{reference}

\begin{thebibliography}{10}

\bibitem{cplex}
{IBM ILOG CPLEX Optimizer}.
\newblock \url{
  http://www.ibm.com/software/commerce/optimization/cplexoptimizer}, 2017.

\bibitem{mosek2017mosek}
The {MOSEK} optimization software.
\newblock \url{http://www. mosek. com}, 2017.

\bibitem{absil2009optimization}
P.-A. Absil, R.~Mahony, and R.~Sepulchre.
\newblock {\em Optimization algorithms on matrix manifolds}.
\newblock Princeton University Press, 2009.

\bibitem{adachi2017solving}
S.~Adachi, S.~Iwata, Y.~Nakatsukasa, and A.~Takeda.
\newblock Solving the {{Trust}}-{{Region Subproblem By}} a {{Generalized
  Eigenvalue Problem}}.
\newblock {\em SIAM Journal on Optimization}, 27(1):269--291, Jan. 2017.

\bibitem{adachi2016eigenvaluebased}
S.~Adachi and Y.~Nakatsukasa.
\newblock Eigenvalue-{{Based Algorithm}} and {{Analysis}} for {{Nonconvex
  QCQP}} with {{One Constraint}}.
\newblock 2016.

\bibitem{attouch2009convergence}
H.~Attouch and J.~Bolte.
\newblock On the convergence of the proximal algorithm for nonsmooth functions
  involving analytic features.
\newblock {\em Mathematical Programming}, 116(1):5--16, 2009.

\bibitem{ben2014hidden}
A.~Ben-Tal and D.~{den Hertog}.
\newblock Hidden conic quadratic representation of some nonconvex quadratic
  optimization problems.
\newblock {\em Mathematical Programming}, 143(1-2):1--29, 2014.

\bibitem{ben2009robust}
A.~Ben-Tal, L.~El~Ghaoui, and A.~Nemirovski.
\newblock {\em Robust Optimization}.
\newblock {Princeton University Press}, 2009.

\bibitem{ben2001lectures}
A.~Ben-Tal and A.~Nemirovski.
\newblock {\em Lectures on Modern Convex Optimization: Analysis, Algorithms,
  and Engineering Applications}, volume~2.
\newblock {Siam}, 2001.

\bibitem{ben1996hidden}
A.~Ben-Tal and M.~Teboulle.
\newblock Hidden convexity in some nonconvex quadratically constrained
  quadratic programming.
\newblock {\em Mathematical Programming}, 72(1):51--63, 1996.

\bibitem{bolte2015error}
J.~Bolte, T.~P. Nguyen, J.~Peypouquet, and B.~W. Suter.
\newblock From error bounds to the complexity of first-order descent methods
  for convex functions.
\newblock {\em Mathematical Programming}, pages 1--37, 2015.

\bibitem{boyd2006subgradient}
S.~Boyd and A.~Mutapcic.
\newblock Subgradient methods.
\newblock {\em Lecture notes of EE364b, Stanford University, Winter Quarter},
  2007, 2006.

\bibitem{boyd2004convex}
S.~Boyd and L.~Vandenberghe.
\newblock {\em Convex Optimization}.
\newblock {Cambridge university press}, 2004.

\bibitem{burer2016how}
S.~Burer and F.~K{\i}l{\i}n\c{}c-Karzan.
\newblock How to convexify the intersection of a second order cone and a
  nonconvex quadratic.
\newblock {\em Mathematical Programming}, pages 1--37, 2016.

\bibitem{conn2000trust}
A.~R. Conn, N.~I. Gould, and P.~L. Toint.
\newblock {\em Trust Region Methods}, volume~1.
\newblock {Society for Industrial and Applied Mathematics (SIAM),
  Philadelphia}, 2000.

\bibitem{feng2012duality}
J.-M. Feng, G.-X. Lin, R.-L. Sheu, and Y.~Xia.
\newblock Duality and solutions for quadratic programming over single
  non-homogeneous quadratic constraint.
\newblock {\em Journal of Global Optimization}, 54(2):275--293, 2012.

\bibitem{flippo1996duality}
O.~E. Flippo and B.~Jansen.
\newblock Duality and sensitivity in nonconvex quadratic optimization over an
  ellipsoid.
\newblock {\em European journal of operational research}, 94(1):167--178, 1996.

\bibitem{fortin2004trust}
C.~Fortin and H.~Wolkowicz.
\newblock The trust region subproblem and semidefinite programming.
\newblock {\em Optimization methods and software}, 19(1):41--67, 2004.

\bibitem{fujie1997semidefinite}
T.~Fujie and M.~Kojima.
\newblock Semidefinite programming relaxation for nonconvex quadratic programs.
\newblock {\em Journal of Global optimization}, 10(4):367--380, 1997.

\bibitem{gao2016ojasiewicz}
B.~Gao, X.~Liu, X.~Chen, and Y.-x. Yuan.
\newblock On the $\{$$\backslash$ L$\}$ ojasiewicz exponent of the quadratic
  sphere constrained optimization problem.
\newblock {\em arXiv preprint arXiv:1611.08781}, 2016.

\bibitem{golub2002inverse}
G.~Golub and Q.~Ye.
\newblock An {{Inverse Free Preconditioned Krylov Subspace Method}} for
  {{Symmetric Generalized Eigenvalue Problems}}.
\newblock {\em SIAM Journal on Scientific Computing}, 24(1):312--334, Jan.
  2002.

\bibitem{gould2010solving}
N.~I.~M. Gould, D.~P. Robinson, and H.~S. Thorne.
\newblock On solving trust-region and other regularised subproblems in
  optimization.
\newblock {\em Mathematical Programming Computation}, 2(1):21--57, Mar. 2010.

\bibitem{cvx}
M.~Grant and S.~Boyd.
\newblock {CVX}: Matlab software for disciplined convex programming, version
  2.1.
\newblock \url{http://cvxr.com/cvx}, Mar. 2014.

\bibitem{guo2009improved}
C.-H. Guo, N.~J. Higham, and F.~Tisseur.
\newblock An improved arc algorithm for detecting definite hermitian pairs.
\newblock {\em SIAM Journal on Matrix Analysis and Applications},
  31(3):1131--1151, 2009.

\bibitem{hazan2016a}
E.~Hazan and T.~Koren.
\newblock A linear-time algorithm for trust region problems.
\newblock {\em Mathematical Programming}, 158(1):363--381, 2016.

\bibitem{hmam2010quadratic}
H.~Hmam.
\newblock Quadratic {{Optimization}} with {{One Quadratic Equality
  Constraint}}.
\newblock Technical report, Warfare and Radar Division DSTO Defence Science and
  Technology Organisation, Australia, Report DSTO-TR-2416, 2010.

\bibitem{ho2017second}
N.~Ho-Nguyen and F.~Kilinc-Karzan.
\newblock A second-order cone based approach for solving the trust-region
  subproblem and its variants.
\newblock {\em SIAM Journal on Optimization}, 27(3):1485--1512, 2017.

\bibitem{Hsia2014revisit}
Y.~Hsia, G.-X. Lin, and R.-L. Sheu.
\newblock A revisit to quadratic programming with one inequality quadratic
  constraint via matrix pencil.
\newblock {\em Pacific Journal of Optimization}, 10(3):461--481, 2014.

\bibitem{huang2016consensus}
K.~Huang and N.~D. Sidiropoulos.
\newblock Consensus-admm for general quadratically constrained quadratic
  programming.
\newblock {\em IEEE Transactions on Signal Processing}, 64(20):5297--5310.

\bibitem{jeyakumar2014trust}
V.~Jeyakumar and G.~Li.
\newblock Trust-region problems with linear inequality constraints: Exact
  {{SDP}} relaxation, global optimality and robust optimization.
\newblock {\em Mathematical Programming}, 147(1-2):171--206, 2014.

\bibitem{jiang2017socp}
R.~Jiang, D.~Li, and B.~Wu.
\newblock {{SOCP}} reformulation for the generalized trust region subproblem
  via a canonical form of two symmetric matrices.
\newblock {\em Mathematical Programming}, pages 1--33, 2017.

\bibitem{li2015new}
G.~Li, B.~S. Mordukhovich, and T.~Pham.
\newblock New fractional error bounds for polynomial systems with applications
  to h{\"o}lderian stability in optimization and spectral theory of tensors.
\newblock {\em Mathematical Programming}, 153(2):333--362, 2015.

\bibitem{li1995error}
W.~Li.
\newblock Error bounds for piecewise convex quadratic programs and
  applications.
\newblock {\em SIAM Journal on Control and Optimization}, 33(5):1510--1529,
  1995.

\bibitem{liu2016quadratic}
H.~Liu, W.~Wu, and A.~M.-C. So.
\newblock Quadratic optimization with orthogonality constraints: Explicit
  lojasiewicz exponent and linear convergence of line-search methods.
\newblock In {\em ICML}, pages 1158--1167, 2016.

\bibitem{luo2000error}
Z.-Q. Luo and J.~F. Sturm.
\newblock Error bounds for quadratic systems.
\newblock In {\em High performance optimization}, pages 383--404. Springer,
  2000.

\bibitem{martinez1994local}
J.~M. Mart{\'\i}nez.
\newblock Local minimizers of quadratic functions on {{Euclidean}} balls and
  spheres.
\newblock {\em SIAM Journal on Optimization}, 4(1):159--176, 1994.

\bibitem{more1993generalizations}
J.~J. Mor{\'e}.
\newblock Generalizations of the trust region problem.
\newblock {\em Optimization Methods and Software}, 2(3-4):189--209, 1993.

\bibitem{More1983Computing}
J.~J. Mor{\'e} and D.~C. Sorensen.
\newblock Computing a trust region step.
\newblock {\em SIAM Journal on Scientific and Statistical Computing},
  4(3):553--572, 1983.

\bibitem{nesterov2003introductory}
Y.~Nesterov.
\newblock {\em Introductory Lectures on Convex Optimization: A Basic Course},
  volume~87.
\newblock Springer Science \& Business Media, 2003.

\bibitem{polik2007survey}
I.~P{\'o}lik and T.~Terlaky.
\newblock A survey of the {{S}}-lemma.
\newblock {\em SIAM Review}, 49(3):371--418, 2007.

\bibitem{pong2014generalized}
T.~K. Pong and H.~Wolkowicz.
\newblock The generalized trust region subproblem.
\newblock {\em Computational Optimization and Applications}, 58(2):273--322,
  2014.

\bibitem{rendl1997semidefinite}
F.~Rendl and H.~Wolkowicz.
\newblock A semidefinite framework for trust region subproblems with
  applications to large scale minimization.
\newblock {\em Mathematical Programming}, 77(1):273--299, 1997.

\bibitem{salahi2016efficient}
M.~Salahi and A.~Taati.
\newblock An efficient algorithm for solving the generalized trust region
  subproblem.
\newblock {\em Computational and Applied Mathematics}, pages 1--19, 2016.

\bibitem{stern1995indefinite}
R.~J. Stern and H.~Wolkowicz.
\newblock Indefinite trust region subproblems and nonsymmetric eigenvalue
  perturbations.
\newblock {\em SIAM Journal on Optimization}, 5(2):286--313, 1995.

\bibitem{sturm2003cones}
J.~F. Sturm and S.~Zhang.
\newblock On cones of nonnegative quadratic functions.
\newblock {\em Mathematics of Operations Research}, 28(2):246--267, 2003.

\bibitem{wang2016linear}
J.~Wang and Y.~Xia.
\newblock A linear-time algorithm for the trust region subproblem based on
  hidden convexity.
\newblock {\em Optimization Letters}, pages 1--8, 2016.

\bibitem{yakubovich1971s}
V.~A. Yakubovich.
\newblock S-procedure in nonlinear control theory.
\newblock {\em Vestnik Leningrad University}, 1:62--77, 1971.

\bibitem{ye1992new}
Y.~Ye.
\newblock A new complexity result on minimization of a quadratic function with
  a sphere constraint.
\newblock In {\em Recent {{Advances}} in {{Global Optimization}}}, pages
  19--31. {Princeton University Press}, 1992.

\bibitem{yuan2015recent}
Y.~Yuan.
\newblock Recent advances in trust region algorithms.
\newblock {\em Mathematical Programming}, 151(1):249--281, 2015.

\bibitem{zhang2010derivativefree}
H.~Zhang, A.~R. Conn, and K.~Scheinberg.
\newblock A derivative-free algorithm for least-squares minimization.
\newblock {\em SIAM Journal on Optimization}, 20(6):3555--3576, 2010.

\bibitem{zhou2017unified}
Z.~Zhou and A.~M.-C. So.
\newblock A unified approach to error bounds for structured convex optimization
  problems.
\newblock {\em Mathematical Programming}, pages 1--40, 2017.

\end{thebibliography}
\end{document}